\documentclass[11pt]{amsart}

\usepackage{a4wide, setspace}
\usepackage{color}
\usepackage{amsmath,amssymb}
\usepackage{amscd}
\usepackage{mathrsfs}
\usepackage{amsthm}

\theoremstyle{plain}

\newtheorem{theorem}{\bf Theorem}[section]
\newtheorem{lemma}[theorem]{\bf Lemma}
\newtheorem{proposition}[theorem]{\bf Proposition}
\newtheorem{corollary}[theorem]{\bf Corollary}

\theoremstyle{definition}
\newtheorem{definition}[theorem]{\bf Definition}

\newtheorem{remark}[theorem]{\bf Remark}

\newtheorem{question}[theorem]{\bf Question}
\newtheorem*{question*}{\bf Question}

\newcommand{\nai}[2]{\langle #1,#2\rangle}

\newcommand{\eqa}[1]{
\begin{align*}
#1
\end{align*}}

\usepackage[all]{xy}

\title[Polish Groups of Finite Type]{Unitarizability, Maurey-Nikishin Factorization, \\ and Polish Groups of Finite Type}

\author{Hiroshi Ando}
\address{Hiroshi Ando, Department of Mathematics and Informatics, Chiba University, 1-33 Yayoi-cho, Inage, Chiba, 263-0022, Japan}
\email{hiroando@math.s.chiba-u.ac.jp}

\author{Yasumichi Matsuzawa}
\address{Yasumichi Matsuzawa, Department of Mathematics, Faculty of Education, Shinshu University, 6 -Ro, Nishinagano, Nagano City 380-8544, Japan}
\email{myasu@shinshu-u.ac.jp}

\author{Andreas Thom} 
\address{Andreas Thom, Fachbereich Mathematik, Institut f\"ur Geometrie, TU Dresden, 01062 Dresden, Germany}
\email{andreas.thom@tu-dresden.de}
\author{Asger T\"ornquist}
\address{Asger T\"ornquist, Department of Mathematical Sciences, University of Copenhagen, Universitetspark 5, 2100 Copenhagen, Denmark}
\email{asgert@math.ku.dk}


\begin{document}

\maketitle

%Theorem 1.3 and Corollary 3.6 corrected: $\mathcal{G}$ must be uniformly bounded in order to make $A'$ in the definition of the equivariant Maurey-Nikishin factorization globally $\mathcal{G}$-invariant.
% Proof of Lemma 3.7 (i) implies (ii) is shortened: there was unnecessary (and incorrect) argument in (i) implies (ii). Since m is a probability measure C(\lambda) is well-defined and is less than or equal to 1. 
\begin{abstract}
Let $\Gamma$ be a countable discrete group, and $\pi\colon \Gamma\to {\rm{GL}}(H)$ be a representation of $\Gamma$ by invertible operators on a separable Hilbert space $H$. 
We show that the semidirect product group $G=H\rtimes_{\pi}\Gamma$ is SIN ($G$ admits a two-sided invariant metric compatible with its topology) and unitarily representable ($G$ embeds into the unitary group $\mathcal{U}(\ell^2(\mathbb N))$), if and only if $\pi$ is uniformly bounded, and that $\pi$ is unitarizable if and only if $G$ is of finite type: that is, $G$ embeds into the unitary group of a II$_1$-factor. Consequently, we show that a unitarily representable Polish SIN group need not be of finite type, answering a question of Sorin Popa.  
The key point in our argument is an equivariant version of the Maurey--Nikishin factorization theorem for continuous maps from a Hilbert space to the space $L^0(X,m)$ of all measurable maps on a probability space.  
\end{abstract}

%\tableofcontents

\section{Introduction}
In this paper, we continue our study \cite{AM12-2} of Popa's class $\mathscr{U}_{\rm{fin}}$ of finite type Polish groups. Here, a Polish group $G$ is said to be of {\it finite type} if it is isomorphic as a topological group to a closed subgroup of the unitary group $\mathcal{U}(M)$ of some finite von Neumann algebra $M$ with separable predual, where $\mathcal{U}(M)$ is endowed with the strong operator topology. 
The class $\mathscr{U}_{\rm{fin}}$ was introduced in \cite{Po07}, in which Popa proved the celebrated Cocycle Superrigidity Theorem. This theorem asserts that if $\Gamma$ is a $w$-rigid group and $\alpha\colon \Gamma\curvearrowright (X,\mu)$ is a probability measure preserving (p.m.p) action of $\Gamma$ on a standard probability space $(X,\mu)$ with certain properties, 
%($s$-malleable and weakly mixing on an infinite normal subgroup $H$ with relative property (T))
then any 1-cocycle $w\colon X\times \Gamma\to G$ for $\alpha$ with values in a Polish group $G\in \mathscr{U}_{\rm{fin}}$ is cohomologous to a group homomorphism $\Gamma \to G$, see \cite{Po07} for details. 
Despite the importance of Popa's Theorem and its wide range of applications, little has been known about the properties of the class $\mathscr{U}_{\rm{fin}}$.    
There are two necessary conditions for a Polish group to be in $\mathscr{U}_{\rm{fin}}$. Namely, if $G\hookrightarrow \mathcal{U}(M)$ for some finite von Neumann algebra $(M,\tau)$ with separable predual, then clearly $G\hookrightarrow \mathcal{U}(\ell^2(\mathbb N))$. $G$ is said to be {\it unitarily representable} if the abovementioned condition is satisfied. Moreover, the restriction of the 2-norm metric $d$ on $\mathcal{U}(M)$ given by 
\[d(u,v):=\|u-v\|_2, \ \ \|x\|_2:=\tau(x^*x)^{\frac{1}{2}}\ (u,v\in \mathcal{U}(M))\]
defines a two-sided invariant metric on $G$ that is compatible with the topology. 
It is well-known that a metrizable topological group admits a two-sided invariant metric compatible with the topology if and only if the group is SIN (see Definition \ref{def: SIN and Ufin}). 
In \cite[$\S$6.5]{Po07}, Popa questioned whether the abovementioned two conditions are sufficient:  
\begin{question*}[Popa]\label{question: Popa} 
Let $G$ be a unitarily representable SIN Polish group. Is $G$ necessarily of finite type? 
\end{question*}
Motivated by the study \cite{AM12-1} of the Lie structure of the class $\mathscr{U}_{\rm{fin}}$, two of us began to research Popa's $\mathscr{U}_{\rm{fin}}$-problem \cite{AM12-2}, and obtained some positive answers when additional conditions were placed on $G$. Among other things, we showed that the question has a positive answer if $G$ is amenable. For example, the Hilbert-Schmidt group $\mathcal{U}(H)_2$ is a finite type Polish group. However, the general case was left unaddressed.

Note that there is a similarity between the $\mathscr{U}_{\rm{fin}}$-problem and the problem of finding a trace on a finite von Neumann algebra, which was established by Murray and von Neumann in their classical papers on the subject.  
Yeadon \cite{Yeadon71} found a relatively simpler proof of Murray's and von Neumann's theorem as follows: Let $M$ be a finite von Neumann algebra. For simplicity, we assume that $M$ has separable predual, and let $\varphi$ be a faithful normal state on $M$. We let the unitary group $\mathcal{U}(M)$ act on the predual $M_*$ by conjugation: \begin{equation}
\alpha_u(\psi):=u\psi u^*=\psi(u^*\cdot u),\ \ \ \ \  \psi\in M_*,\, u\in \mathcal{U}(M).
\end{equation}
Note that $\alpha$ is a linear isometric action on the Banach space $M_*$. Using the finiteness assumption, Yeadon showed that the orbit $\mathcal{O}_{\varphi}:=\{\alpha_u(\varphi); u\in \mathcal{U}(M)\}$ is relatively weakly compact. Thus, its weak convex closure $\overline{\text{co}}(\mathcal{O}_{\varphi})\subset M_*$ is a weakly compact convex subset of $M_*$ on which $\mathcal{U}(M)$ acts by affine isometries. By the Ryll-Nardzewski Theorem, this implies that there exists an $\alpha$-fixed point, which is easily seen to be a normal faithful tracial state $\tau$.  

In the first named author's PhD thesis, we adopted the Yeadon argument up to a certain point. This part of the thesis formed an unpublished joint work with the second named author. The argument was as follows. Let $G$ be a unitarily representable Polish SIN group. It is well-known (see e.g. \cite{Gao05}) that $G$ is unitarily representable if and only if $G$ admits a continuous positive definite function $\varphi$ that generates the topology. An analogous characterization for $\mathscr{U}_{\rm{fin}}$ is obtained in \cite{AM12-2}. Let $\mathcal{P}(G)$ be the set of all normalized continuous positive definite functions on $G$. Then $G$ is of finite type if and only if there exists a character (i.e., a function $\varphi\in \mathcal{P}(G)$ which is conjugation-invariant) that generates a neighborhood basis of the identity on $G$ (see Theorem \ref{thm: characterization of finite type} for details). It is clear that $\mathcal{P}(G)$ is a convex subset of the Banach space $C_{\rm{b}}(G)$ of all continuous bounded functions on $G$ equipped with the sup norm. 
Define an affine isometric action $\beta\colon G\curvearrowright C_{\rm{b}}(G)$ by:
\begin{equation*}
\beta_g(\psi):=\psi(g^{-1}\cdot g), \ \ \ \ \ \ \psi\in C_{\rm{b}}(G),\, g\in G.
\end{equation*}
If the orbit $\mathcal{O}_{\varphi}=\{\beta_g(\varphi); g\in G\}$ is relatively weakly compact, then by the  Ryll-Nardzewski Theorem, there exists a fixed point $\psi$ in the closed convex hull $\overline{\text{co}}(\mathcal{O}_{\varphi})$. Once we get a fixed point $\psi\in \overline{\text{co}}(\mathcal{O}_{\varphi})$, one can show, using Theorem \ref{thm: characterization of finite type} below, that $G$ is of finite type. 
Unfortunately, we do not know if $\mathcal{O}_{\varphi}$ is relatively weakly compact. What we could show is the following result.
\begin{theorem}[Ando--Matsuzawa, unpublished] 
Let $G$ be a unitarily representable Polish {\rm SIN} group, and assume that $\varphi\in \mathcal{P}(G)$ generates the topology of $G$. 
Then the following statements hold: 
\begin{itemize}
\item[{\rm{(i)}}] $\mathcal{O}_{\varphi}$ is relatively compact in the compact open topology. 
\item[{\rm{(ii)}}] If $\mathcal{O}_{\varphi}$ satisfies one of the following conditions, then $G$ is of finite type. 
\begin{itemize}
\item[{\rm{(iia)}}] The closure of $\mathcal{O}_{\varphi}$ with respect to the topology of pointwise convergence is norm-separable. 
\item[{\rm{(iib)}}] $\mathcal{O}_{\varphi}$ is relatively weakly compact. 
\end{itemize}   
\end{itemize}
\end{theorem} 
Two of the current authors considered the abovementioned result as an indication that the $\mathscr{U}_{\rm{fin}}$-problem may have a positive answer in general. 
However, there was no obvious means of verifying the relation between the compact open topology and the weak topology in $C_{\rm{b}}(G)$ or verifying either (iia) or (iib).   

In this paper, we show that indeed the orbit $\mathcal{O}_{\varphi}\subset C_{\rm{b}}(G)$ is not always relatively weakly compact. 
That is, we construct a family of unitarily representable SIN Polish groups that do not fall into the class $\mathscr{U}_{\rm{fin}}$. This answers Popa's abovementioned question. The construction is based on the existence of uniformly bounded non-unitarizable representations of countable discrete groups.
\begin{theorem}[Main Theorem]\label{thm: main}
Let $\Gamma$ be a countable discrete group, and let $\pi\colon \Gamma\to {\rm{GL}}(H)$ be a group homomorphism from $\Gamma$ to the group of bounded invertible operators on $H$. Let $G=H\rtimes_{\pi}\Gamma$ be the semidirect product. 
\begin{itemize}
\item[{\rm{(a)}}] 
$G$ is a unitarily representable Polish group. Moreover, the following two conditions are equivalent. 
\begin{itemize}
\item[{\rm{(i)}}] $G$ is a {\rm SIN} group. 
\item[{\rm{(ii)}}] $\pi$ is uniformly bounded. That is, $\|\pi\|:=\sup_{s\in \Gamma} \|\pi(s)\|<\infty$ holds. 
\end{itemize}
\item[{\rm{(b)}}] Assume that $\pi$ is uniformly bounded. Then the following three conditions are equivalent. 
\begin{itemize}
\item[{\rm{(i)}}] $G$ is of finite type. 
\item[{\rm{(ii)}}] There exists a character $f\colon H\to \mathbb{C}$ that generates a neighborhood basis of $0\in H$, such that $f(\pi(s)\xi)=f(\xi)$ for all $\xi\in H$ and $s\in \Gamma$. 
\item[{\rm{(iii)}}]$\pi$ is unitarizable. That is, there exists $V\in {\rm{GL}}(H)$ such that $V^{-1}\pi(s)V$ is unitary for all $s\in \Gamma$. 
\end{itemize}
\end{itemize}
\end{theorem}
A group $\Gamma$ is said to be {\it unitarizable} if all of its uniformly bounded representations are unitarizable. 
The problem of unitarizability has a long history. Sz.-Nagy \cite{Nagy47} showed that the group $\mathbb{Z}$ of integers is unitarizable. 
Later, this was independently generalized to amenable groups by Dixmier \cite{Dixmier50}, Nakamura--Takeda \cite{NakamuraTakeda51} and Day \cite{Day50}.
This led Dixmier to pose the question whether every unitarizable group is amenable, a question that remains open today. 
The first examples of non-unitarizable groups were found by Ehrenpreis and Mautner \cite{EhrenpreisMautner55}, who showed that SL$_2(\mathbb{R})$ is not unitarizable. For discrete groups, explicit families of uniformly bounded non-unitarizable representations of free groups were later constructed by Mantero--Zappa \cite{ManteroZappa83,ManteroZappa83-2}, Pitlyk--Szwarc \cite{PitlykSzwarc86}, and Bozejko \cite{Bozejko87}. We refer the reader to \cite{Pisier01,Pisier05,Ozawa06} for details. Because unitarizability passes to subgroups, it follows that all discrete groups containing free subgroups are non-unitarizable. We also note that examples of non-unitarizable groups not containing non-abelian free subgroups were presented by Epstein--Monod \cite{EpsteinMonod09}, Osin \cite{Osin09} and   Monod--Ozawa \cite{MonodOzawa10}. Therefore, there exist many uniformly bounded representations that are not unitarizable. The present article provides a new and unexpected connection between the study of unitarizability of groups and the understanding of the class $\mathscr{U}_{\rm{fin}}$.  

Finally, we explain the idea of the proof. 
Theorem \ref{thm: main} (a) is shown in Theorem \ref{thm: SIN iff uniformly bounded}. For (b), (i)$\Leftrightarrow$(ii) is Theorem \ref{thm: characterization of Ufin semidirect product} and (iii)$\Rightarrow $(i) is Corollary \ref{cor: unitarizable implies Ufin}. Finally, (i)$\Rightarrow$(iii) is shown in Theorem \ref{thm: keythm}. The proofs of (a) and (b) (i)$\Leftrightarrow $(ii) and (iii)$\Rightarrow$(i) rely on our characterization \cite{AM12-2} of the class $\mathscr{U}_{\rm{fin}}$ via positive definite functions. 
The most difficult step is to prove that (b) (i)$\Rightarrow $(iii). Note that our argument is different from that of Vasilescu-Zsido \cite{VasilescuZsido74}, who used the Ryll-Nardzewski Theorem, and showed that if a uniformly bounded representation $\pi$ generates a finite von Neumann algebra, then $\pi$ is unitarizable. 
From the assumption that $G=H\rtimes_{\pi}\Gamma$ is of finite type, there appears to be no obvious means to deduce the conditoin that $W^*(\pi(\Gamma))$ is a finite von Neumann algebra (observe that when $\pi$ is an irreducible unitary representation, then $H\rtimes_{\pi}\Gamma$ is of finite type, but $\pi(\Gamma)''=\mathbb{B}(H)$ is not a finite von Neumann algebra if $\dim(H)=\infty$).   
Assume that $G=H\rtimes_{\pi}\Gamma$ is of finite type. 
Then there exists $f\in \mathcal{P}(H)$ that generates the neighborhood basis of 0 and satisfies $f(\pi(s)\xi)=f(\xi)\ (s\in \Gamma, \xi\in H)$. If $\dim(H)<\infty$, then by Bochner's Theorem, there exists a probability measure $\mu$ on $H$ (regarded as a real Hilbert space) such that 
\[f(\xi)=\int_He^{ix\cdot \xi}\,\text{d}\mu(x),\ \ \ \ \ \xi\in H.\]
If $\dim(H)=\infty$, Bochner's Theorem does not hold (e.g. the positive definite function $\xi\mapsto e^{-\|\xi\|^2}$ is {\it not} the Fourier transform of a probability measure on $H$). However, there exists a cylinder set probability measure $\mu$ defined on the algebraic dual $H_a'$ of $H$, such that 
\begin{equation}
f(\xi)=\int_{H_a'}e^{i\nai{\chi}{\xi}}\,\text{d}\mu(\chi),\ \ \ \ \ \xi\in H.\label{eq: representation of f by Bochner}
\end{equation}
It can, in fact, be shown that $\mu$ is supported on a Hilbert-Schmidt extension of $H$. See \cite{Yamasaki} for details. Intuitively, one may hope that $\pi$ is unitarizable by the following argument: define 
\begin{equation*}
\nai{\xi}{\eta}_{\pi}':=\int_{H_a'}\nai{\chi}{\xi}\nai{\chi}{\eta}\,\text{d}\mu(\chi),\ \ \ \ \ \xi,\eta\in H,
\end{equation*}
which is (if well-defined) an $\mathbb{R}$-bilinear form. Because $f$ is $\pi$-invariant, its dual action $\tilde{\pi}$ of $\Gamma$ on $H_a'$ preserves $\mu$, which then implies that $\nai{\cdot}{\cdot}_{\pi}'$ is $\Gamma$-invariant: $\nai{\pi(s)\xi}{\pi(s)\eta}_{\pi}'=\nai{\xi}{\eta}_{\pi}'\ (\xi,\eta\in H,\ s\in \Gamma)$. By complexification, we obtain a $\pi(\Gamma)$-invariant inner product $\nai{\cdot}{\cdot}_{\pi}$. We may then hope that the inner product provides the same topology as the original inner product, which would imply that $\pi$ is unitarizable (see Lemma \ref{lem: Dixmier}). 
There are two difficulties in this approach: (a)  we must guarantee that second moments exist: $$\int_{H_a'}\nai{\chi}{\xi}^2\ \text{d}\mu(\chi)\stackrel{?}{<}\infty$$ and (b) we must prove that the new inner product is equivalent to the original one. It does not seem likely that these conditions hold in general. 
We remedy this by noncommutative integration theory: assume that $\alpha\colon G\to \mathcal{U}(M)$ is an embedding into the unitary group of a II$_1$-factor $M$. Then by setting $\alpha(s):=\alpha(0,s)$ and $\alpha(\xi):=\alpha(\xi,1)\ (\xi\in H,\ s\in \Gamma)$, we obtain $\alpha(s)\alpha(\xi)\alpha(s)^*=\alpha(\pi(s)\xi).$ 
Let $\widetilde{M}$ be the space of all unbounded closed operators affiliated with $M$.  
Because $t\mapsto \alpha(t\xi)$ is a strongly  continuous one-parameter unitary group for each $\xi\in H$, there exists a self-adjoint operator $T(\xi)\in \widetilde{M}_{\rm{sa}}$ such that $\alpha(t\xi)=e^{itT(\xi)}\ (t\in \mathbb{R})$. 
One can show that $T\colon H\to \widetilde{M}_{\rm{sa}}$ is an $\mathbb{R}$-linear homeomorphism onto its range. 
If we can show that $T(\xi)\in L^2(M,\tau)$ for every $\xi\in H$ (we call it the {\it $L^2$-condition}), then
\begin{equation}
\nai{\xi}{\eta}_{\pi}':=\nai{T(\xi)}{T(\eta)}_{L^2(M,\tau)},\ \ \ \ (\xi,\eta\in H)\label{eq: invariant inner product, almost}
\end{equation}
is a well-defined inner-product, which is $\Gamma$-invariant by $\alpha(s)T(\xi)\alpha(s)^*=T(\pi(s)\xi)\ (\xi\in H, s\in \Gamma)$. We can also show that the induced norm is equivalent to the original norm. By complexification, we are done by Lemma \ref{lem: Dixmier}. 
Again, the real question is whether the $L^2$-condition holds in general, a question essentially equivalent to postulating the existence of the second moments of $\mu$. However, because the von Neumann algebra $\mathcal{A}=\{\alpha(\xi);\,\xi\in H\}''$ generated by the image of $\alpha$ is commutative, we may represent it as $\mathcal{A}=L^{\infty}(X,m)$ for some probability space $(X,m)$ equipped with an $m$-preserving $\Gamma$-action $\beta$ such that the W$^*$-dynamical system  $(\widetilde{\mathcal{A}},\text{Ad}(\alpha(\ \cdot\ )))$ is equivariantly $*$-isomorphic to $(L^{0}(X,m), \beta)$. Here,  $L^0(X,m)$ is the space of all measurable maps on a probability space $(X,m)$ equipped with the topology of convergence in measure 
Because $H$ is a Hilbert space (hence of Rademacher type 2, see \cite[$\S$6]{KaltonAlbiac06}), it is known (by the so-called Maurey--Nikishin factorization \cite{Nikishin70,Maurey72}) that any continuous linear operator $T\colon H\to L^0(X,m)=\widetilde{A}$ factors through a Hilbert space. That is, there exists a bounded operator $\tilde{T}\colon H\to L^2(X,m)$ and an $m$-a.e.\ positive function $\psi\in L^{0}(X,m)$ such that the following diagram commutes: 
\[
\xymatrix{
H\ar[rr]^{T}\ar@{}[rrd]|{\circlearrowleft}\ar[dr]_{\tilde{T}} & & L^0(X,m)\\
&L^2(X,m)\ar[ur]_{M_{\psi}} &
}
\]
Here, $M_{\psi}$ is the multiplication operator by the function $\psi$. 
More precisely, there exists an $m$-a.e. positive $\varphi\in L^{\infty}(X,m)$ such that 
$\int_X \varphi(x) [T\xi](x)^2\ \text{d}m(x)<\infty$ for every $\xi\in H$. Then $\tilde{T}\xi:=M_{\varphi^{\frac{1}{2}}}\cdot T\xi$ and $\psi=\varphi^{-\frac{1}{2}}$ gives the factorization. 
This result was first essentially proved by Nikishin \cite{Nikishin70}. Nikishin's proof was later simplified by Maurey \cite{Maurey72} using the theory of Rademacher types. We refer the reader to \cite{Pisier86,KaltonAlbiac06,GCRF85} for more details. In $\S$3, we show an equivariant version of the Maurey--Nikishin factorization, in which we use a convexity argument to ensure (Theorem \ref{thm: Gamma-inv Nikisihin}) that the operator $\tilde{T}$ can be made $\Gamma$-equivariant if $T$ is assumed to be $\Gamma$-equivariant. We can then define $\nai{\cdot}{\cdot}_{\pi}'$ by the formula (\ref{eq: invariant inner product, almost}), with $T$ replaced by $\tilde{T}$. 
Because $T$ is a real linear homeomorphism onto its range, we can show that $\tilde{T}$ is also a real linear homeomorphism onto its range too. The last step is to show that on the image $\tilde{T}(H)$ of $H$ under $\tilde{T}$, the topology of the convergence in measure agrees with the $L^2$-topology (Lemma \ref{lem: L^2homeo}). This is somewhat surprising: note that these two topologies do not agree on the whole of $L^2(X,m)$. The proof is done by using the fact \cite{Beltita,AM12-1} that on $\widetilde{M}$, the strong resolvent topology and the $\tau$-measure topology of Nelson \cite{Nelson74} agree.  Once this is established, it is straightforward to show that $\nai{\cdot}{\cdot}_{\pi}'$ generates the same topology as that of the original inner product, and therefore that $\pi$ is unitarizable. 

Finally, we report that as a byproduct of the existence of an equivariant Maurey--Nikishin factorization, we find a new, useful description of positive definite functions on a Hilbert space (compare (\ref{eq: representation of f by Bochner}) with (\ref{eq: description of f})). See also the comments following Corollary \ref{cor: Cor to MaureyNikishin}.  
\begin{theorem}[Corollary \ref{cor: Cor to MaureyNikishin}] Let $H$ be a separable real Hilbert space, and let $f\in \mathcal{P}(H)$. Then there exist 
\begin{itemize}
\item[{\rm{(i)}}] a compact metrizable space $X$ and a Borel probability measure $m$ on $X$,
\item[{\rm{(ii)}}] a bounded linear operator $T\colon H\to L^2(X,m;\mathbb{R})$, and 
\item[{\rm{(iii)}}] a positive measurable function $\varphi\in L^{0}(X,m)$, 
\end{itemize}
such that 
\begin{equation}\label{eq: description of f}
f(\xi)=\int_Xe^{i\varphi(x)[T\xi](x)}\ {\rm{d}}m(x),\ \ \  \ \ \ \xi\in H.\end{equation}
Moreover:
\begin{itemize}
\item[{\rm{(iv)}}] If $f$ generates a neighborhood basis of $0$, then $T$ can be chosen to be a real-linear homeomorphism onto its range. 
\item[{\rm{(v)}}] We may arrange $T$ and $\varphi$ such that if $\mathcal{G}$ is a uniformly bounded subgroup of the group $\{u\in {\rm{GL}}(H); f(u\xi)=f(\xi)\ (\xi\in H)\}$ of invertible operators on $H$ fixing $f$ pointwise, then there exists a continuous homomorphism $\beta$ from $\mathcal{G}$ to the group ${\rm{Aut}}(X,m)$ of all $m$-preserving automorphisms of $X$ with the weak topology, such that 
\begin{equation*}
[T(u\xi)](x)=\beta(u)[T(\xi)](x),\ \ \ \varphi(\beta(u)x)=\varphi(x),\ \ \ \ \ u\in \mathcal{G},\ \xi\in H,\ m\text{-a.e.\ }x\in X.
\end{equation*} 
\end{itemize}
\end{theorem}
\section{The class $\mathscr{U}_{\rm{fin}}$ and uniformly bounded representations} 
In this section we recall Popa's class $\mathscr{U}_{\rm{fin}}$ of finite type Polish groups and consider when a specific semidirect group belongs to this class. 
The results of this section are extracted from the first named author's PhD thesis, jointly proved with the second named author. 
\begin{definition}Let $G,G'$ be topological groups. 
\begin{itemize}
\item[(i)] We say that $G$ is {\it embeddable} into $G'$, if there exists a topological group isomorphism $\pi$ from $G$ onto a closed subgroup of $G'$. Such $\pi$ is called an {\it embedding}. 
\item[(ii)] We denote by $\mathcal{P}(G)$ the set of all normalized continuous positive definite functions on $G$, and by $\mathcal{P}(G)_{\rm{inv}}$ the set of all $f\in \mathcal{P}(G)$ which is conjugation-invariant, i.e.,  $f(g^{-1}xg)=f(x)$ holds for all $x,g\in G$. An element of $\mathcal{P}(G)_{\rm{inv}}$ is called a {\it character} on $G$. 
\end{itemize} 
\end{definition}
 
\begin{definition}[\cite{Po07}]
A Polish group $G$ is of {\it finite type} if $G$ is embeddable into the unitary group $\mathcal{U}(M)$ of some finite von Neumann algebra $M$ equipped with strong operator topology. The class of finite type Polish groups is denoted by $\mathscr{U}_{\rm{fin}}$. 
\end{definition}
If $G$ is a finite type Polish group, then by \cite[Lemma 2.6]{Po07}, the above $M$ may chosen to be a type II$_1$-factor with separable predual. 

\begin{definition}\label{def: SIN and Ufin}
Let $G$ be a topological group. 
\begin{itemize}
\item[(i)] $G$ is called {\it unitarily representable} if $G$ is embeddable into $\mathcal{U}(\ell^2(\mathbb N))$ equipped with strong operator topology. 
\item[(ii)] $G$ is called {\it SIN} (small invariant neighborhood) if for any neighborhood $V$ of $G$ at the identity $e$, there exists a neighborhood $U\subset V$ of $e$ which is conjugation invariant, i.e., $g^{-1}Ug=U,\ \forall g\in G$. Such $U$ is called an {\it invariant neighborhood} of $G$. 
\end{itemize}
If $G$ is metrizable, then $G$ is SIN if and only if $G$ admits a bi-invariant metric $d$ which is compatible with its topology.
\end{definition}
For more details about unitary representability, see e.g., \cite{Galindo09,Gao05,Megrelishvili02}. We will use the following characterization of the finite type Polish groups (cf. \cite{Gao05}). 
\begin{theorem}[\cite{AM12-2}]\label{thm: characterization of finite type}
For a Polish group $G$, the following conditions are equivalent.

\begin{itemize}
\item[{\rm{(i)}}] $G$ is of finite type.
%\item[{\rm{(ii)}}] $G$ is isomorphic as a topological group onto a closed subgroup of the unitary group of a finite von Neumann algebra acting on a separable Hilbert space. 
\item[{\rm{(ii)}}] There exists a family $\mathcal{F}\subset \mathcal{P}(G)_{\rm{inv}}$ which 
generates a neighborhood basis of the identity $e_G$ of $G$.
That is, for each neighborhood $V$ at the identity, there are functions $f_1,\dots,f_n\in \mathcal{F}$
and open sets $\mathcal{O}_1,\dots, \mathcal{O}_n$ in $\mathbb{C}$ such that
\begin{equation*}
e_G\in\bigcap_{i=1}^{n}f_i^{-1}(\mathcal{O}_i)\subset V.
\end{equation*} 
\item[{\rm{(iii)}}] There exists $f\in \mathcal{P}(G)_{\rm{inv}}$ that generates a neighborhood basis of the identity of $G$.
%$f$ such that for each neighborhood $V$ at the identity there exists an open set $\mathcal{O}$ in $\mathbb{C}$ with $e_G\in f^{-1}(\mathcal{O})\subset V$.
\item[{\rm{(iv)}}] There exists a family $\mathcal{F}\subset \mathcal{P}(G)_{\rm{inv}}$ which  separates the identity of $G$ and closed subsets $A$ with $A\not\ni e_G$.
That is, for each closed subset $A$ with $A\not\ni e_G$, 
there exists $f\in\mathcal{F}$ such that
\begin{equation*}
\sup_{x\in A}|f(x)| < |f(e_G)|.
\end{equation*}
\item[{\rm{(v)}}] There exists $f\in \mathcal{P}(G)_{\rm{inv}}$ which separates the identity of $G$ 
and closed subsets $A$ with $A\not\ni e_G$.
\end{itemize}
 
\end{theorem}
We start from the following characterization for the semidirect product of a SIN group by a SIN group to be SIN. Denote by $\mathscr{N}(G)$\index{$\mathscr{N}(G)$ the set of all neighborhoods of the identity in $G$} the set of all neighborhoods of identity of a topological group $G$.
\begin{theorem}\label{thm: characterization of SIN groups} Let $N,K$ be {\rm SIN}-groups. Let $\alpha\colon K\to {\rm{Aut}}(N)$ be a continuous action of $K$ on $N$. 
Then the following two conditions are necessary and sufficient for $N\rtimes_{\alpha}K$ to  be a {\rm SIN}-group.
\begin{list}{}{}
\item[{\rm{(SIN. 1)}}] For any $U\in \mathscr{N}(N)$, there exist $V\in \mathscr{N}(N)$ with $V\subset U$ such that 
\[\alpha_k(V)=V,\ \ \ \ k\in K.\]
\item[{\rm{(SIN. 2)}}] The action $\alpha$ is bounded. That is, for any $V\in \mathscr{N}(N)$, there exists $W\in \mathscr{N}(K)$ such that 
\[n^{-1}\alpha_k(n)\in V,\ \ \ \ n\in N,\ k\in W.\]
\end{list} 
\end{theorem}
\begin{proof}We denote by $e_N\ (\text{resp.}\ e_K)$ the unit of $N\ (\text{resp.}\ K)$.

We first prove sufficiency of the two conditions. Fix an arbitrary $U\in \mathscr{N}(N\rtimes_{\alpha}K)$. By  definition of the product topology, there exist $V''\in \mathscr{N}(N)$ and $W'\in \mathscr{N}(K)$ such that $(e_N,e_K)\in V''\times W'\subset U$ holds.  Since $N$ is SIN, there exists an invariant neighborhood $V'\in \mathscr{N}(N)$ which satisfies $V'\cdot V'\subset V''$. By (SIN. 1), we may take $V\in \mathscr{N}(N)$ such that $V\subset V'$ and 
\[\alpha_k(V)=V,\ \ \ \ k\in K.\]
By (SIN. 2) and the SIN property of $K$, there exists an invariant neighborhood $W\in \mathscr{N}(K)$ satisfying $W\subset W'$ and 
\[n^{-1}\alpha_k(n)\in V,\ \ \ \ n\in N,\ k\in W.\]
Indeed, by (SIN. 2) there is $W_0\in \mathscr{N}(K)$ such that $n^{-1}\alpha_k(n)\in V$ holds for all $n\in N$ and $k\in W_0$. By the SIN property of $K$, there is an invariant neighborhood $W\in \mathscr{N}(K)$ contained in $W_0\cap W'$. It is then clear that this $W$ satisfies the above requirement. 
  
Let $(n_1,k_1)\in V\times W$ and $(n_2,k_2)\in N\rtimes_{\alpha}K$. We see that
\eqa{
(n_2,k_2)(n_1,k_1)(n_2,k_2)^{-1}&=(n_2\alpha_{k_2}(n_1),k_2k_1)(\alpha_{k_2^{-1}}(n_2^{-1}),k_2^{-1})\\
&=(n_2\alpha_{k_2}(n_1)\alpha_{k_2k_1k_2^{-1}}(n_2^{-1}),k_2k_1k_2^{-1})\\
&=\left (\{n_2\alpha_{k_2}(n_1)n_2^{-1}\}\{n_2\alpha_{k_2k_1k_2^{-1}}(n_2^{-1})\},k_2k_1k_2^{-1}\right )\\
&\in (n_2Vn_2^{-1}\cdot V)\times W\\
&\subset (n_2V'n_2^{-1}\cdot V')\times W'=(V'\cdot V')\times W'\\
&\subset U. 
}
Hence, we have $g^{-1}(V\times W)g\subset U$ for all $g\in N\rtimes_{\alpha}(K)$. This shows 
\[U':=\bigcup_{g\in N\rtimes_{\alpha}K}g^{-1}(V\times W)g\subset U\]
is an invariant neighborhood of $e_{N\rtimes_{\alpha}H}$ contained in $U$. Since $U$ is arbitrary, this shows $N\rtimes_{\alpha}K$ is a SIN-group. 

We now prove necessity of the two conditions. Suppose $N\rtimes_{\alpha}K$ is a SIN-group. Given $U\in \mathscr{N}(N)$. Then $U\times K\in \mathscr{N}(N\rtimes_{\alpha}K)$. Thus by SIN-property of $N\rtimes_{\alpha}K$ we find an invariant neighborhood $A\in \mathscr{N}(N\rtimes_{\alpha}K)$ contained in $U\times K$. Then we may find $U_0\in \mathscr{N}(N)$ and $W\in \mathscr{N}(K)$ such that $U_0\times W\subset A$. Since $A$ is invariant, we have 
\[g(U_0\times \{e_K\})g^{-1}\subset A,\ \ \ \ g\in N\rtimes_{\alpha}K.\]
Inparticular, for all $k\in K$ and $n\in U_0$, we have 
\[(e_N,k)(n,e_K)(e_N,k)^{-1}=(\alpha_k(n),e_K)\in A\subset U\times K.\]
This shows $\alpha_k(n)\in U$ for all $k\in K$ and $n\in U_0\subset U$. Then 
\[V:=\bigcup_{k\in K}\alpha_k(U_0)\subset U\]
is a neighborhood of $e_N$ satisfying $\alpha_k(V)=V$ for all $k\in K$. Hence we conclude that (SIN. 1) holds. 

In order to prove (SIN. 2), suppose $V\in \mathscr{N}(N)$ is given. Again by the same argument as above, we may find an invariant neighborhood $A\in \mathscr{N}(N\rtimes_{\alpha}K)$ contained in $V\times K$ and $V_0\in \mathscr{N}(N)$ and $W\in \mathscr{N}(K)$ with $V_0\times W\subset A\subset V\times K$. Since $A$ is invariant, we have in particular
\[g(\{e_N\}\times W)g^{-1}\subset A\subset V\times K,\ \ \ \ g\in N\rtimes_{\alpha}K.\]
Hence for $n\in N$ and $k\in W$, we have  
\[(n,e_K)(e_N,k)(n,e_K)^{-1}=(n\alpha_k(n^{-1}),k)\in A\subset V\times K.\]
This shows $n\alpha_k(n^{-1})\in V$ for all $k\in W$. This proves (SIN. 2).  
\end{proof}
\begin{remark}\label{rem: SIN 2 is automatic for discrete group action}
A semidirect product $G\rtimes_{\alpha}\Gamma$ of a SIN group $G$ by a discrete group $\Gamma$ always satisfies (SIN. 2) (take $W=\{e_K\}$).
\end{remark}
As an immediate corollary to Theorem \ref{thm: characterization of SIN groups}, we obtain: 
\begin{theorem}\label{thm: SIN iff uniformly bounded}
Let $\Gamma$ be a countable discrete group, $H$ be a separable Hilbert space and $\pi\colon \Gamma\rightarrow {\rm{GL}}(H)$ be an invertible representation of $\Gamma$.
Then the semidirect product group $G:=H\rtimes_{\pi}\Gamma$ is unitarily representable. Moreover, the following conditions are equivalent.
\begin{itemize}
\item[{\rm{(1)}}]  $G$ is {\rm SIN}.
\item[{\rm{(2)}}] $\pi$ is uniformly bounded, i.e., $\|\pi\|:=\sup_{s\in \Gamma}\|\pi(s)\|<\infty$.
\end{itemize}
\end{theorem}
\begin{proof}
The unitarily representable group $H$ has a positive definite function $f(\xi):=e^{-\|\xi\|^2},\ \xi \in H$ and identifying $H$ with $H\times \{e\}\subset H\rtimes_{\pi}\Gamma$, we may extend $f$ to be a continuous positive definite function on $H\rtimes_{\pi}\Gamma$ by taking value 0 outside $H\times \{e\}$. It is immediate to see that this extended function still generates the neighborhood basis of the identity. Then by \cite[Theorem 2.1]{Gao05}, $G$ is unitarily representable.\\ 
(1)$\Rightarrow $(2) Assume that $G$ is a unitarily representable and SIN group. Let $B_{\delta}(0)$ be the open $\delta$-neighborhood of $H$ at 0\ $(\delta>0)$. Then by (SIN. 1), there exists a neighborhood $V\subset B_{\delta}(0)$ of $H$ at 0 such that $\pi(s)V=V,\ (s\in \Gamma)$.  Since $V$ is a neighborhood, there exists $\varepsilon>0$ such that $B_{\varepsilon}(0)\subset V$ holds. Thus $\pi(s)B_{\varepsilon}(0)\subset B_{\delta}(0)$ holds for all $s\in \Gamma$. This shows that $\|\pi(s)\|\le \frac{\delta}{\varepsilon}$ for every $s\in \Gamma$. Therefore $\pi$ is uniformly bounded.
\\
(2)$\Rightarrow $(1)  Assume that $\pi$ is uniformly bounded. To prove that $G$ is SIN, by Theorem \ref{thm: characterization of SIN groups} we have only to check (SIN.1) (see Remark \ref{rem: SIN 2 is automatic for discrete group action}). 
Consider $B_{\varepsilon}(0)\ (\varepsilon>0)$, and $C:=1+\sup_{s\in\Gamma}\|\pi(s)\|<\infty$. 
Then for all $\xi\in B_{\varepsilon/C}(0)$, $s\in\Gamma$, we have 
\[
\|\pi(s)\xi\| \leq C\|\xi\| < \varepsilon.
\]
This implies $\pi(s)B_{\varepsilon/C}(0)\subset B_{\varepsilon}(0)$
for all $s\in \Gamma$. Then $V:=\bigcup_{s\in \Gamma}\pi(s)B_{\varepsilon/C}(0)$ satisfies  (SIN.1).
\end{proof}
\begin{theorem}\label{thm: characterization of Ufin semidirect product}
Let $\Gamma$ be a countable discrete group, $H$ be a separable Hilbert space and $\pi:\Gamma\rightarrow \mathbb{B}(H)$ be a uniformly bounded\index{uniformly bounded representation} representation of $\Gamma$.
Then $H\rtimes_{\pi}\Gamma$ is of finite type if and only if
there exists $f\in \mathcal{P}(H)$ such that
\begin{itemize}
\item[{\rm{(F.1)}}] $f(\pi(s)\xi) = f(\xi),\ \ \ \ 
\forall s\in\Gamma,\ \xi\in H$,

\item[{\rm{(F.2)}}] $f$ generates a neighborhood basis of $0\in H$.
\end{itemize}
\end{theorem}

\begin{proof}
We first prove ``if'' part.
For this, we assume that there exists a positive, continuous positive definite function $f$ on $H$ which satisfies (F.1) and (F.2).
Then one can extend $f$ on $H\rtimes_{\pi}\Gamma$ by
\[
\tilde{f}(\xi,s) := 
\begin{cases}
\ \ \ \ f(\xi) & (s=1), \\
\ \ \ \ 0 & (s\not= 1).
\end{cases}
\]
It is easy to see that $\tilde{f}$ is continuous, positive definite and
generates a neighborhood basis of the identity.
Take arbitrary $\xi,\eta\in H$, $s,t\in\Gamma$.
Since
\[
(\eta,t)(\xi,s)(\eta,t)^{-1}
=\left(\eta+\pi(t)\xi-\pi(tst^{-1})\eta,tst^{-1}\right),
\]
we have
\begin{align*}
\tilde{f}\left((\eta,t)(\xi,s)(\eta,t)^{-1}\right)
&= \begin{cases}
\ \ \ \ \tilde{f}\left(\pi(t)\xi,1\right) & (s=1), \\
\ \ \ \ 0 & (s\not= 1).
\end{cases}\\
&= \begin{cases}
\ \ \ \ f(\xi) & (s=1), \\
\ \ \ \ 0 & (s\not= 1).
\end{cases}\\
&=\tilde{f}(\xi,s).
\end{align*}
Hence $H\rtimes_{\pi}\Gamma$ is of finite type.

Conversely, let us assume that $H\rtimes_{\pi}\Gamma$ is of finite type.
Then there exists a positive, continuous positive definite class function $\tilde{f}$ on $H\rtimes_{\pi}\Gamma$ that generates a neighborhood basis of the identity.
We define a positive definite function $f$ on $H$ by
\[
f(\xi) := \tilde{f}(\xi,1), \ \ \ \ \xi\in H.
\]
It is easy to check that $f$ is positive, continuous and satisfies (F.2).
Since
\[
(\eta,t)(\xi,1)(\eta,t)^{-1}
=\left(\pi(t)\xi,1\right)
\]
and $\tilde{f}$ is a class function,
we obtain
\begin{align*}
f(\xi) &= \tilde{f}(\xi,1) 
= \tilde{f}\left((\eta,t)(\xi,1)(\eta,t)^{-1}\right)\\
&= \tilde{f}\left(\pi(t)\xi,1\right)
= f\left(\pi(t)\xi\right)
\end{align*}
so that $f$ satisfies (F.1).
\end{proof}

\begin{corollary}\label{cor: unitarizable implies Ufin}
Let $\Gamma$ be a countable discrete group, $H$ be a separable Hilbert space and $\pi:\Gamma\rightarrow \mathbb{B}(H)$ be a uniformly bounded representation of $\Gamma$.
If $\pi$ is unitarizable, then $H\rtimes_{\pi}\Gamma$ is of finite type.
\end{corollary}

\begin{proof}
Since $\pi$ is unitarizable\index{unitarizable}, there exists an invertible operator $T\in\mathbb{B}(H)$ such that $T\pi(s)T^{-1}$ is unitary for every $s\in\Gamma$.
Define a positive definite function $f$ on $H$ by 
\[
f(\xi) := e^{-\|T\xi\|^2}, \ \ \ \ \xi\in H.
\]
Then $f$ is positive, continuous and satisfies (F.2).
Moreover
\[
f(\pi(s)\xi) = e^{-\|T\pi(s)\xi\|^2} 
= e^{-\left\|\left\{T\pi(s)T^{-1}\right\}T\xi\right\|^2}
= e^{-\|T\xi\|^2}
= f(\xi)
\]
so that (F.1) holds.
\end{proof}

Since amenable groups are unitarizable (\cite{Day50,Dixmier50,NakamuraTakeda51}), we obtain: 
\begin{corollary}\label{cor: amenable case}
Let $\Gamma$ be a countable discrete group, $H$ be a separable Hilbert space and $\pi:\Gamma\rightarrow \mathbb{B}(H)$ be a uniformly bounded representation of $\Gamma$.
Assume that $\Gamma$ is amenable.
Then $H\rtimes_{\pi}\Gamma$ is of finite type.
\end{corollary}
\begin{remark}
We remark that if $\Gamma$ is amenable, then so is $H\rtimes_{\pi}\Gamma$, so that the conclusion of Corollary \ref{cor: amenable case} also follows from the fact \cite[Theorem 2.2]{AM12-2} that an amenable unitarily representable SIN group is of finite type. 
\end{remark}
The purpose of the next section is to prove a converse to Corollary \ref{cor: unitarizable implies Ufin}. 
\section{Equivariant Maurey--Nikishin Factorization}
% and Proof of ``$G=H\rtimes_{\pi}\Gamma\in \mathscr{U}_{\rm{fin}}\Rightarrow \pi$ unitarizable"}
In this section, we prove the last and crucial step, Theorem \ref{thm: main} (b) (i)$\Rightarrow $(iii). 
Let $\pi\colon \Gamma \to H$ be a uniformly bounded representation of a discrete countable group $\Gamma$ on a separable Hilbert space $H$, as before.  
\begin{theorem}\label{thm: keythm} 
If $G=H\rtimes_{\pi}\Gamma$ is of finite type, then $\pi$ is unitarizable. 
\end{theorem}
Throughout this section, we assume that $G=H\rtimes_{\pi}\Gamma$ is of finite type. Our aim is to find a $\Gamma$-invariant inner product and apply the following classical result due to Dixmier: 
\begin{lemma}[Dixmier]\label{lem: Dixmier}
Let $\pi\colon \Gamma\to \mathbb{B}(H)$ be a uniformly bounded representation. Then $\pi$ is unitarizable if and only if there exists an inner product $\nai{\cdot}{\cdot}'$ on $H$ that generates the same Hilbert space topology and which is $\pi$-invariant, i.e., $\nai{\pi(s)\xi}{\pi(s)\eta}'=\nai{\xi}{\eta}'\ (\xi,\eta\in H, s\in \Gamma)$ holds. 
\end{lemma}
By assumption, there exists an embedding (as a topological group) $\alpha\colon G\to \mathcal{U}(M)$, where $M$ is some II$_1$-factor with separable predual and with the normal faithtul tracial state $\tau$. We let $M$ act on $L^2(M,\tau)$.  
As usual, we identify $H\subset G\subset \mathcal{U}(M)$ and $\Gamma\subset \mathcal{U}(M)$. That is, $\xi\in H$ is identified with $\alpha(\xi)=\alpha(\xi,0)$, and $s\in \Gamma$ is identified with $\alpha(s)=\alpha(0,s)$. Therefore under this identification, the uniformly bounded $G$-action $\pi$ on $H$ is implemented by unitaries in $M$:
\begin{equation}
\alpha(\pi(s)\xi)=\alpha((0,s)(\xi,1)(0,s^{-1})=\alpha(s)\alpha(\xi)\alpha(s)^*.
\end{equation}
Now, let $\widetilde{M}$ be the space of all (possibly unbounded) closed and densely defined operators $T$ on $L^2(M,\tau)$ which are affiliated with $M$. 
There is a completely metrizable topology on $\widetilde{M}$, called the {\it $\tau$-measure topology} \cite{Nelson74}. A neighborhood basis of 0 in the $\tau$-measure topology is given by the family $\{N(\varepsilon,\delta);\varepsilon,\delta>0\}$, where 
\[N(\varepsilon,\delta):=\{A\in \widetilde{M};\ \exists p\in \text{Proj}(M), \tau(p^{\perp})<\delta,\ \|Ap^{\perp}\|<\varepsilon\},\]
where $\text{Proj}(M)$ is the lattice of all projections of $M$. 
%Here, $\|Ap^{\perp}\|<\varepsilon$ for unbounded $A$ means that the range of $p^{\perp}=1-p$ is contained in the domain $\dom{A}$ of $A$, and $Ap^{\perp}\in \mathbb{B}(L^2(M))$ has norm $<\varepsilon$. 
In \cite{AM12-1}, we showed that the $\tau$-measure topology coincides with the strong resolvent topology (SRT) (the same result is obtained by Beltita \cite{Beltita} using the standard form). In particular, if $A_n,A\in \widetilde{M}$ are self-adjoint, then $A_n\stackrel{n\to \infty}{\to}A$ in the $\tau$-measure topology, if and only if $\|(A_n-i)^{-1}-(A-i)^{-1}\|_2\stackrel{n\to \infty}{\to}0$, if and only if (by the Kato-Trotter Theorem) $\|e^{itA_n}-e^{itA}\|_2\stackrel{n\to \infty}{\to} 0$ for every $t\in \mathbb{R}$ (in this case the convergence is compact uniform in $t$). As was shown by Murray-von Neumann, the space $\widetilde{M}$, although not locally convex, is actually a topological $*$-algebra with respect to the $\tau$-measure topology. The addition (resp. the multiplication) is defined by $(A,B)\mapsto \overline{A+B}$ (resp. $(A,B)\mapsto \overline{AB}$), where the bar denotes the operator closure (it asserts that $A+B$, $AB$ are always closable and in fact $\overline{A+B}$ is self-adjoint whenever $A,B$ are). 

Now let $\xi\in H$. Then since $\alpha$ is continuous, $\mathbb{R}\ni t\mapsto \alpha(t\xi)\in \mathcal{U}(M)$ is a strongly continuous one-parameter unitary group. Thus by Stone's Theorem, there exists a unique self-adjoint operator $T(\xi)\in \widetilde{M}$ such that 
\[\alpha(t\xi)=e^{itT(\xi)},\ \ \ \ \ t\in \mathbb{R},\ \xi\in H.\]

\begin{proposition}\label{prop: covariance}
The map $T\colon H\to \widetilde{M}_{\rm{sa}}$ is a real-linear homeomorphism onto its range (with respect to the $\tau$-measure topology or the {\rm{SRT}}). 
Moreover, the following covariance property holds. 
\begin{equation}
T(\pi(s)\xi)=\alpha(s)T(\xi)\alpha(s)^*,\ \ \ \ s\in \Gamma,\xi\in H.\label{eq: T(pi(g))}
\end{equation}
\end{proposition}
\begin{proof}
Let $s\in \Gamma$ and $\xi_n,\xi,\eta\in H\ (n\in \mathbb{N})$. 
Since $\alpha$ is a topological isomorphism onto its range, 
\eqa{
\xi_n-\xi\stackrel{n\to \infty}{\to}0&\Leftrightarrow \alpha(t\xi_n)-\alpha(t\xi)=e^{itT(\xi_n)}-e^{itT(\xi)}\stackrel{n\to \infty}{\to} 0\ \text{\ for\ all\ }t\in \mathbb{R}\ (\text{SOT})\\
&\Leftrightarrow T(\xi_n)\stackrel{n\to \infty}{\to}  T(\xi)\ (\text{SRT}).
} 
Thus $T$ is a homeomorphism from $H$ onto its range. 
Also, for all $s,t\in \mathbb{R}$, we have
\[\alpha(ts\xi)=e^{itsT(\xi)}=\alpha(t(s\xi))=e^{itT(s\xi)}.\]
Since $t$ is arbitrary, the uniqueness part of Stone's Theorem asserts that $T(s\xi)=sT(\xi)$. 
Now since $H$ is a commutative (additive) group, all $\{T(\xi);\xi\in H\}$ are strongly commuting self-adjoint operators. Thus by Kato-Trotter Theorem, 
\eqa{
e^{it\overline{T(\xi)+T(\eta)}}&=\text{SOT}-\lim_{n\to \infty}(e^{it\frac{1}{n}T(\xi)}e^{it\frac{1}{n}T(\eta)})^{n}=e^{itT(\xi)}e^{itT(\eta)}\\
&=\alpha(t\xi)\alpha(t\eta)=\alpha(t(\xi+\eta))\\
&=e^{itT(\xi+\eta)}.
}
Therefore we obtain $T(\xi+\eta)=\overline{T(\xi)+T(\eta)}$. Therefore $T$ is a real-linear homeomorphism onto its range.
Finally, by the unitary covariance of the Borel functional calculus,  
\eqa{
e^{itT(\pi(s)\xi)}&=\alpha(t\pi(s)\xi)=\alpha(s)\alpha(t\xi)\alpha(s)^*\\
&=\alpha(s)e^{itT(\xi)}\alpha(s)^*\\
&=e^{it\alpha(s)T(\xi)\alpha(s)^*}.
}
Therefore $T(\pi(s)\xi)=\alpha(s)T(\xi)\alpha(s)^*$. 
\end{proof}

Let $\mathcal{A}:=\{\alpha(\xi);\xi\in H\}''$ be the commutative von Neumann subalgebra of $M$ generated by the image of $\alpha$. By Proposition \ref{prop: trace is integration}, there exists a probability space $(X,m)$ such that $(\mathcal{A},\tau|_{\mathcal{A}})\cong (L^{\infty}(X,m),\ \int_X \cdot\ dm)$. Then by  (\ref{eq: T(pi(g))}), the formula 
\begin{equation}
\beta_s:=\text{Ad}(\alpha(s))|_{\mathcal{A}},\ \ \ \ s\in \Gamma.
\end{equation}
defines a trace-preserving action of $\Gamma$ on $\mathcal{A}=L^{\infty}(X,m)$. This action also defines a probability measure preserving action on $X$. We denote this action by $\Gamma \times X\ni (s,x)\mapsto s\cdot x\in X$. Therefore, $\beta_s(\varphi)=\varphi\circ s^{-1}, \varphi \in L^{\infty}(X,m)$. 
Below we keep this setting. 
\begin{definition}
We denote by $L^0(X,m)$ the space of all $m$-measurable $\mathbb{C}$-valued functions on $X$, and $L^0(X,m)_+$ its positive cone. We also define the $\Gamma$-action on $L^0(X,m)$ by $\beta_s=\text{Ad}(\alpha(s))|_{L^0(X,m)}\ (s\in \Gamma)$ (we use the same symbol $\beta_s$ as the $\Gamma$-action on $L^{\infty}(X,m)$, but there would be no danger of confusion). 
\end{definition}
Equipped with the convergence in measure, $L^0(X,m)$ is a (non-locally convex) completely metrizable topological vector space. 
Note that under the identification $\mathcal{A}=L^{\infty}(X,m)$, we obtain $\widetilde{\mathcal{A}}=L^0(X,m)\subset \widetilde{M}_{\rm{sa}}$ (cf. \cite{AM12-1}). 
What we will do is to construct a $\Gamma$-equivariant Maurey--Nikishin factorization of the map $T\colon H\to L^0(X,m)\subset \widetilde{M}_{\rm{sa}}$. We almost exactly follow Maurey's argument as in \cite[Chapter VI.3]{GCRF85}, but we also need to take care of the $\Gamma$-action, which makes it necessary to refine some of the arguments.  

\begin{theorem}[Equivariant Maurey-Nikishin Factorization] \label{thm: Gamma-inv Nikisihin} 
The map $T\colon H\to L^0(X,m)$  factors $\Gamma$-equivariantly through a Hilbert space. More precisely, there exists a $\Gamma$-invariant function $\varphi\in L^{0}(X,m)$ which is $m$-a.e. positive and
\begin{equation}
\int_X \varphi (x)[T\xi](x)^2\, {\rm{d}}m(x)\le \|\xi\|^2,\ \ \ \ \ \ \xi \in H.\label{eq: Equivariant Nikishin}
\end{equation}
In particular, the following diagram commutes: 
\[
\xymatrix{
H\ar[rr]^{T}\ar@{}[rrd]|{\circlearrowleft}\ar[dr]_{\tilde{T}} & & L^0(X,m)\\
&L^2(X,m)\ar[ur]_{M_{\psi}} &
}
\]
Here, $M_{\psi}$ is the multiplication operator by the $\Gamma$-invariant function $\psi=\varphi^{-\frac{1}{2}}$ and $\tilde{T}\xi:=M_{\varphi^{\frac{1}{2}}}\cdot T\xi$ is $\Gamma$-equivariant. 
\end{theorem}

\begin{definition}
For $\varphi \in L^0(X,m)$, we define its {\it non-increasing rearrangement} $\varphi^{\star}\colon [0,1]\to [0,\infty)$ by 
\begin{equation*}
\varphi^{\star}(t):=\inf \{\lambda>0; m(\{x\in X; |\varphi (x)|>\lambda\})\le t\},\ \ \ \ t\in [0,1].
\end{equation*}
Then clearly $(c\varphi)^{\star}(t)=c\varphi^{\star}(t)\ (c>0,\ t\ge 0)$ holds. 
\end{definition}
The next lemmata are elementary. 
\begin{lemma}\label{lem: boundedness in measure}
Let $S\colon H\to L^0(X,m)$ be a linear operator. Then the following conditions are equivalent.
\begin{itemize}
\item[{\rm{(i)}}] $S$ is continuous. 
\item[{\rm{(ii)}}] There exists a non-increasing map $C\colon (0,\infty)\to [0,1]$ such that $\lim_{\lambda\to \infty}C(\lambda)=0$ and 
\begin{equation}
m(\{x\in X; |S\xi (x)|\ge \lambda \|\xi\|\})\le C(\lambda)\ \ \ \ \ \ (\lambda>0,\ \xi\in H).\label{eq: C(lambda)}
\end{equation}
\item[{\rm{(iii)}}] For every $t>0$, there exists $K(t)>0$ such that 
\begin{equation}
(S\xi)^{\star}(t)\le K(t)\|\xi\|\ \ \ \ (\xi\in H).\label{eq: continuity of S}
\end{equation}
\end{itemize}
\end{lemma}
\begin{proof}
(i)$\Rightarrow $(ii) For each $\varepsilon,\ \delta>0$, the set $U(\varepsilon,\delta):=\{f\in L^0(X,m); m(\{x\in X; |f(x)|\ge \varepsilon\})\le \delta\}$ is an neighborhood of 0 in the measure topology. For each $\lambda>0$, define 
\[C(\lambda):=\inf \{\delta>0; S\xi\in U(\lambda,\delta)\text{\ for\ all\ }\xi\in H\text{\ with\ }\|\xi\|\le 1\}.\] 
Note that $C(\lambda)\le 1$.  
It is then clear that $C(\cdot)$ is nonincreasing and (\ref{eq: C(lambda)}) holds. 
We claim that $\lim_{\lambda\to \infty}C(\lambda)=0$. Assume by contradiction that there exists $\delta>0$ such that $C(\lambda)\ge \delta$ holds for all $\lambda>0$. Then for each $n\in \mathbb{N}$, there exists $\xi_n\in H$ with $\|\xi_n\|\le 1$, such that $S\xi_n\notin U(n,\delta-\frac{1}{n})$, i.e., $m(\{x\in X; |S\xi_n(x)|\ge n\})>\delta-\frac{1}{n}$ holds. Then $\eta_n:=\frac{1}{n}\xi_n\ (n\in \mathbb{N})$ satisfies $\|\eta_n\|\stackrel{n\to \infty}{\to}0$, but 
$m(\{x\in X; |S\eta_n(x)|\ge 1)>\delta-\frac{1}{n}$, so $S\eta_n\not\to 0$ in measure, a contradiction.\\
(ii)$\Rightarrow$(iii) Let $t>0$. Since $\lim_{\lambda\to \infty}C(\lambda)=0$, there exists $K(t)>0$ such that $C(\lambda)<t\ (\lambda\ge K(t))$ holds. Then for $\xi\in H$ with $\|\xi\|\le 1$, 
$m(\{x\in X; |S\xi(x)|>\lambda\})\le C(\lambda)$, so that $(S\xi)^{\star}(t)\le K(t)$. For general $\xi\neq 0$, we have 
$(S\|\xi\|^{-1}\xi)^{\star}(t)\le \|\xi\|^{-1}K(t)$, whence (\ref{eq: continuity of S}) holds.\\
(iii)$\Rightarrow $(i) Since $H$ and $L^0(X,m)$ are metrizable, it suffices to show the sequential continuity of $S$. Let $\xi_n\in H$ be such that $\|\xi_n\|\stackrel{n\to \infty}{\to}0$. Let $\varepsilon,\delta>0$. 
By assumption, we have $(S\xi_n)^{\star}(\delta)\le K(\delta)\|\xi_n\|\stackrel{n\to \infty}{\to}0$. 
Therefore there exists $N\in \mathbb{N}$ such that $(S\xi_n)^{\star}(\delta)<\varepsilon$ holds for all $n\ge N$. In particular, we have $m(\{x\in X; |S\xi_n(x)|\ge \varepsilon\})\le \delta$. Since $\varepsilon,\delta>0$ are arbitrary, $S\xi_n$ converges to 0 in measure. 
\end{proof}
\begin{lemma}\label{lem: easy lemma}
Let $\varphi \in L^0(X,m)$. Then for every $\delta>0$, $m(\{x\in X; |\varphi (x)|>\varphi^{\star}(\delta)\})\le \delta$ holds. 
\end{lemma}
\begin{proof} Let $\lambda:=\varphi^{\star}(\delta)$. Then for each $n\in \mathbb{N}$, there exists $\lambda_n\in (\lambda,\lambda+\frac{1}{n})$ such that 
$\mu(\{x\in X; |\varphi(x)|>\lambda_n\})\le \delta$. Thus 
\begin{align*}
m(\{x\in X; |\varphi(x)|>\lambda\})&\le \lim_{n\to \infty}m(\{x\in X; |\varphi(x)|>\lambda+\tfrac{1}{n}\})\\
&\le \limsup_{n\to \infty}m(\{x\in X; |\varphi(x)|>\lambda_n\})\\
&\le \delta.
\end{align*} 
\end{proof}

\if0
\begin{lemma}
Let $S\colon H\to L^0(X,m)$ be a continuous linear operator. 
%Define for each $(\xi_j)_{j=1}^{\infty}\in \ell^2(H)$ and $x\in X$ a sequence $\widetilde{S}((\xi_n)_{n=1}^{\infty})(x):=(S\xi_n(x))_{n=1}^{\infty}$. 
%Then $(S\xi_j(x))_{j=1}^{\infty}\in \ell^2(\mathbb{N})$ for $m$-a.e. $x\in X$, so that we can define $\widetilde{S}\colon \ell^2(H)\to L_{\ell^2}^0(X,m)$. 
The following conditions are equivalent. 
\begin{itemize}
%\item[{\rm{(i)}}] $\widetilde{S}$ is continuous.  
\item[{\rm{(i)}}] There exists a nonincreasing function $C\colon (0,\infty)\to [0,\infty)$ such that for every finite sequence $\xi_1,\dots,\xi_n\in H$ with $\sum_{j=1}^n\|\xi_j\|^2\le 1$, one has 
\begin{equation}
m\left (\left \{x\in X; \left (\sum_{j=1}^n(S\xi_j)(x)^2\right )^{\frac{1}{2}}\ge \lambda \right \}\right )\le C(\lambda).
\end{equation}
\item[{\rm{(ii)}}] For every $t>0$, there exists $K(t)>0$ such that for any finite sequence $\xi_1,\dots,\xi_n\in H$, one has 
\begin{equation}
\left (\sum_{j=1}^n(S\xi_j)^2\right )^{\star}(t)\le K(t)^2\sum_{j=1}^{n}\|\xi_j\|^2.
\end{equation} 
\end{itemize}  
\end{lemma}
%\begin{proof}
%Let $\xi=(\xi_j)_{j=1}^{\infty}\in \ell^2(H)$. 
%For each $k\in \mathbb{N}$, define 
%$X_k:=\bigcap_{n=1}^{\infty}X_{k,n}$, where $X_{k,n}:=\{x\in X; \sum_{j=1}^{\infty}|S\xi_j(x)|^2<k^2\}$. 
%Let $\widetilde{X}:=\{x\in X; \sum_{j=1}^{\infty}|S\xi_j(x)|^2<\infty\}=\bigcup_{k=1}^{\infty}X_k$. 
%The proof is very similar to Lemma \ref{lem: boundedness in measure}. So we omit the proof.
%\end{proof}
\fi 

The crucial step in our proof is to establish the following proposition. This is the step which allows to make the construction of the factorization canonical and hence equivariant with respect to a group of symmetries.
\begin{proposition}\label{prop: GCRF3.1} Let $A$ be a convex subset of $L^0(X,m)_+$ which is globally $\Gamma$-invariant. Then for every $\varepsilon>0$, there is a $\Gamma$-invariant measurable subset $S(\varepsilon)\subset X$ with $m(X\setminus S(\varepsilon))\le \varepsilon$ and 
\begin{equation}
\sup_{\varphi \in A}\int_{S(\varepsilon)}\varphi(x)\,{\rm{d}}m(x)\le 2\sup_{\varphi \in A}\varphi^{\star}(\tfrac{\varepsilon}{2}).\label{eq: sup_varphi}
\end{equation}
\end{proposition}
This is a $\Gamma$-invariant version of \cite[Proposition VI.3.1]{GCRF85}. 
In order to show Proposition \ref{prop: GCRF3.1}, we use the following variant of the Minimax Lemma:
\begin{lemma}[Minimax Lemma]\label{lem: minimax} 
Let $X,Y$ be topological vector spaces, and let $A$ (resp. $B$) be a convex subset of $X$ (resp. $Y$). Assume that $B$ is compact. If the function $\Phi\colon A\times B\to \mathbb{R}\cup \{+\infty\}$ satisfies$\colon$
\begin{itemize}
\item[{\rm{(i)}}] $\Phi(\cdot,b)$ is a concave function on $A$ for each $b\in B$,
\item[{\rm{(ii)}}] $\Phi(a,\cdot)$ is a convex function on $B$ for each $a\in A$,
\item[{\rm{(iii)}}] $\Phi(a,\cdot)$ is lower-semicontinuous on $B$ for each $a\in A$, 
\end{itemize}
then the following identity holds:
\begin{equation*}
\min_{b\in B}\sup_{a\in A}\Phi(a,b)=\sup_{a\in A}\min_{b\in B}\Phi(a,b).
\end{equation*}
\end{lemma}  
\begin{proof} See \cite[Appendix A.2]{GCRF85}.
\end{proof}
\begin{proof}[Proof of Proposition \ref{prop: GCRF3.1}]
Define 
\begin{equation*}
B:=\left \{\varphi \in L^0(X,m); 0\le \varphi\le 1\ (m\text{-a.e.}), \int_X\varphi(x)\,\text{d}m(x)\ge 1-\frac{\varepsilon}{2}\right \}.
\end{equation*}
Then we first check that $B$ is a globally $\Gamma$-invariant, weakly compact, convex subset of the Hilbert space $L^2(X,m)$: indeed, $B\subset L^{\infty}(X,m)\subset L^2(X,m)$ is clearly a convex set contained in the unit ball of $L^2(X,m)$ (any $\varphi\in B$ satisfies $\int_X |\varphi|^2dm\le \int_X1dm=1$). Since $m\circ \beta_s=m\ (s\in \Gamma)$, the $\Gamma$-invariance is clear. Suppose $(\varphi_n)_{n=1}^{\infty}$ is a sequence in $B$ converging weakly to $\varphi\in L^2(X,m)$. For each $k\in \mathbb{N}$, let  $E_k=\{x\in X; \varphi(x)\le -\frac{1}{k}\}$. Then 
\[0\le \lim_{n\to \infty}\nai{\varphi_n}{1_{E_k}}_{L^2(X)}=\nai{\varphi}{1_{E_k}}_{L^2(X)}=\int_{E_k}\varphi\ \text{d}m\le -\frac{1}{k}m(E_k)\le 0,\]
so that $m(E_k)=0$. Therefore $E=\{x\in X; \varphi(x)<0\}=\bigcup_{k\in \mathbb{N}}E_k$ satisfies $m(E)=0$, so $\varphi\ge 0$ a.e. Similarly, $\varphi\le 1$ a.e.  holds, and 
\[\int_X\varphi\ \text{d}m=\nai{\varphi}{1}_{L^2(X)}=\lim_{n\to \infty}\nai{\varphi_n}{1}\ge 1-\frac{\varepsilon}{2}.\]  
Therefore $\varphi\in B$. This shows that $B$ is weakly closed, so that $B$ is a weakly compact convex subset of $L^2(X,m)$.  

Next, define $\Phi\colon A\times B\to \mathbb{R}\cup \{+\infty\}$ by 
\begin{equation*}
\Phi(\psi,\varphi):=\int_X \psi \varphi\ \text{d}m,\ \ \ \ \ \ (\psi\in A,\ \varphi\in B).
\end{equation*}
It is clear that $\Phi$ satisfies conditions (i) and (ii) of Lemma \ref{lem: minimax}. To verify the condition (iii) of Lemma \ref{lem: minimax}, define for $\psi\in A,\ \varphi\in B$ and $t>0$: 
\begin{align*}
E_t(\psi):&=\{x\in X; \psi(x) \le t\},\\
\Phi_t(\psi,\varphi)&:=\int_{E_t(\psi)}\psi \varphi\ \text{d}m. 
\end{align*} 
Then for $\psi\in A$ and $t>0$, the map $\Phi_t(\psi,\cdot)\colon B\to \mathbb{R}_+$ is weakly continuous, and for $\varphi\in B$, we have 
\begin{equation*}
\Phi(\psi,\varphi)=\sup_{t>0}\Phi_t(\psi,\varphi).
\end{equation*}
This shows that $\Phi(\psi,\cdot)$ is weakly lower-semicontinuous on $B$. 
Let $M:=\sup_{\psi\in A}\psi^{\star}(\frac{\varepsilon}{2})$, and assume that $M<\infty$ (if $M=\infty$, there is nothing to prove). Then for every $\psi\in A$, 
\eqa{
1-\int_{X}1_{E_M(\psi)}\ \text{d}m&=m(\{x\in X; |\psi(x)|>M\})\\
&\le m(\{x\in X;\ |\psi(x)|>\psi^{\star}(\tfrac{\varepsilon}{2})\})\le \tfrac{\varepsilon}{2}.
} 
Here we used Lemma \ref{lem: easy lemma}. This shows that $1_{E_M(\psi)}\in B$ for every $\psi\in A$. Therefore by Lemma \ref{lem: minimax}, 
\begin{align}
\min_{\varphi \in B}\sup_{\psi \in A}\int_X\psi \varphi\ \text{d}m&=\sup_{\psi \in A}\min_{\varphi \in B}\int_X\psi \varphi\ \text{d}m\notag \\
&\le \sup_{\psi\in A}\int_{E_M(\psi)}\psi\ \text{d}m\le M.\label{eq: min varphi}
\end{align}
Therefore the minimum of the left hand side will be attained at some $\varphi\in B$. So the set $B_0$ of all such minimizers $\varphi\in B$ is nonempty. We observe that $B_0$ is a globally $\Gamma$-invariant weakly closed, convex subset of $B$. 
Indeed, let $0<\lambda<1$ and $\varphi_1,\varphi_2\in B_0$ be minimizers and let $\varphi=\lambda \varphi_1+(1-\lambda)\varphi_2$. 
Then for $\psi\in A$, 
\eqa{
\sup_{\psi \in A}\int_X \psi \varphi\ \text{d}m&\le \lambda \sup_{\psi \in A}\int_X\psi \varphi_1\ \text{d}m+(1-\lambda)\sup_{\psi \in A}\int_X\psi \varphi_2\ \text{d}m\\
&=\min_{\varphi'\in B}\sup_{\psi \in A}\int_X\psi \varphi'\ \text{d}m.
}
This shows that $\varphi$ is also a minimizer, hence in $B_0$. Next, let $(\varphi_n)_{n=1}^{\infty}$ be a sequence of elements in $B_0$ converging weakly to $\varphi\in L^2(X,m)$. Then $\varphi\in B$, and Fatou's Lemma shows 
\begin{align*}
\sup_{\psi \in A}\int_X \psi \varphi\ \text{d}m&\le \sup_{\psi\in A}\liminf_{k\to \infty}\int_{X}\underbrace{\psi 1_{\{x; \psi(x)\le k\}}}_{\in L^2(X,m)}\varphi\ \text{d}m\\
&=\sup_{\psi\in A}\liminf_{k\to \infty}\lim_{n\to \infty}\int_{X}\psi 1_{\{x; \psi(x)\le k\}}\varphi_n\ \text{d}m\\
&\le \sup_{\psi\in A}\liminf_{k\to \infty}\limsup_{n\to \infty}\int_X\psi \varphi_n\ \text{d}m\\
&\le \sup_{\psi\in A}\limsup_{n\to \infty}\sup_{\psi'\in A}\int_X \psi'\varphi_n\ \text{d}m\\
&\stackrel{\varphi_n\in B_0}{=}\limsup_{n\to \infty}\min_{\varphi'\in B}\sup_{\psi'\in A}\int_X\psi'\varphi'\ \text{d}m\\
&=\min_{\varphi'\in B}\sup_{\psi'\in A}\int_X\psi' \varphi'\ \text{d}m.
\end{align*} 
This shows that $\varphi$ is a minimizer too, so $\varphi\in B_0$. 
Finally, let $\varphi\in B_0$ and $s\in \Gamma$. 
Then since $A$ is globally $\Gamma$-invariant and $m$ is $\Gamma$-invariant, we have (recall that the $\Gamma$ action on $X$ is given so that  $\beta_s(\varphi)=\varphi(s^{-1}\cdot\ )$)
\begin{align}
\sup_{\psi\in A}\int_X\psi \beta_s(\varphi)\ \text{d}m&=\sup_{\psi \in A}\int_X\beta_{s^{-1}}(\psi)\varphi\ \text{d}m(g\cdot) \notag\\
&=\sup_{\psi \in \beta_{s^{-1}}(A)}\int_X\psi \varphi\ \text{d}m \notag\\
&=\min_{\varphi'\in B}\sup_{\psi \in A}\int_X\psi \varphi'\ \text{d}m. \notag
\end{align} 
This shows that $\beta_s(\varphi)\in B_0$. 
Therefore $B_0$ is a $\Gamma$-invariant, non-empty, weakly compact, and convex subset of $L^2(X,m)$. Let $\varphi_0$ be an element in $B_0$ of minimum $L^2$-norm. By the uniform convexity, such an element is unique. 
Then since $\|\beta_s(\varphi_0)\|_{L^2(X)}=\|\varphi_0\|_{L^2(X)}\ (s\in \Gamma)$, we have $\beta_s(\varphi_0)=\varphi_0\ (s\in \Gamma)$ by the uniqueness of $\varphi_0$. Then define 
\begin{equation*}
S(\varepsilon):=\{x\in X;\ \varphi_0(x)\ge \tfrac{1}{2}\}.
\end{equation*}
Then $S(\varepsilon)$ is a $\Gamma$-invariant measurable subset of $X$. 
Moreover, 
\begin{align}
m(X\setminus S(\varepsilon))&=m(\{x\in X; \varphi_0(x)<\tfrac{1}{2}\}) \notag\\
&\le 2\int_X (1-\varphi_0(x))\ \text{d}m\le \varepsilon,\notag\\
\sup_{\psi\in A}\int_{S(\varepsilon)}\psi\ \text{d}m&\le 2\sup_{\psi \in A}\int_X\psi \varphi_0\ \text{d}m\le 2M.\notag
\end{align}
Here in the last line we used (\ref{eq: min varphi}).    
\end{proof}

In the sequel, we will mainly follow the arguments in \cite[Theorem VI.3.3]{GCRF85}. Since this source is hardly accessible, we decided to repeat  some of the arguments (and present them in a streamlined form) in order to provide a complete proof.
We set $Z:= \{ |(T\xi)(x)|^2; \xi \in H, \|\xi\| \leq 1 \}\subset L^0(X,m)_+$ and note that 
\begin{equation}
A:=\left\{ \sum_{j=1}^n |(T\xi_j)|^2; n \in \mathbb N, \xi_1,\dots,\xi_n \in H, \sum_{j=1}^n \|\xi_j\|^2 \leq 1 \right\}\subset L^0(X,m)_+\label{eq: def of A}
\end{equation} is the convex hull of $Z$. In order to apply Proposition \ref{prop: GCRF3.1} to (a subset of) $A$, we need to show that the estimates from Lemma \ref{lem: boundedness in measure} extend to the convex hull of $Z$. This is the content of the following proposition.

\begin{proposition}\label{prop: 3.3 (d)}
There exists a function $(0,\infty)\ni \lambda \mapsto C(\lambda)\in [0,1]$ satisfying $$\lim_{\lambda\to \infty}C(\lambda)=0,$$ such that for every finite family $\xi_1,\cdots,\xi_n\in H$, one has 
\begin{equation*}
\sum_{j=1}^n\|\xi_j\|^2\le 1\Rightarrow m\left (\left \{x\in X; \left (\sum_{j=1}^n|T\xi_j(x)|^2\right )^{\frac{1}{2}}\ge\lambda \right \}\right )\le C(\lambda).
\end{equation*} 
\end{proposition}
This is Condition (d) in \cite[Theorem VI.3.3]{GCRF85}, and the proof can be found in \cite[Theorem VI.3.6]{GCRF85}. 
We give an alternative and direct proof, following a strategy that we learned from G.\ Lowther \cite{Lowther} on MathOverflow. 
First, we need a classical Khintchine type inequality in $L^0$-space. We fix a {\it Rademacher sequence} $(\epsilon_n)_{n=1}^{\infty}$ on a probability space $(\Omega,\mathbb{P})$. That is, $(\epsilon_n)_{n=1}^{\infty}$ is a sequence of independent, identically distributed random variables on $(\Omega,\mathbb{P})$ with the distribution $\mathbb{P}(\epsilon_n=1)=\frac{1}{2}=\mathbb{P}(\epsilon_n=-1)\ (n\in \mathbb{N})$. The symbol $\mathbb{E}$ denotes the expectation value with respect to $\mathbb{P}$. 
\begin{lemma}\label{lem: KhinchinL0} For all $n \in \mathbb N$ and for all $x = (x_j)_{j=1}^n\in \mathbb R^n$, we have
$$
\mathbb{P}\left(\left(\sum_{j=1}^n\epsilon_j x_j\right)^2\ge \frac12 \cdot \sum_{j=1}^n x_j^2\right)\ge \frac{1}{12}.
$$
\end{lemma}
\begin{proof} Without loss of generality, we may assume that $\sum_{j=1}^n x_j^2 =1$. 
We consider the random variable $f(\epsilon_1,\dots,\epsilon_n) = \left(\sum_{j=1}^n \epsilon_j x_j \right)^2$. Then, using independence, we compute $\mathbb E(f(\epsilon_1,\dots,\epsilon_n)) = \sum_{j=1}^n x_j^2=1$ and
\begin{equation*}
\mathbb E(f(\epsilon_1,\dots,\epsilon_n)^2) = \sum_{j=1}^n x_j^4 + 3 \sum_{i \neq j} x_i^2 x_j^2 = 3 \left( \sum_{j=1}^n x_j^2 \right)^2 - 2 \sum_{j=1}^n x_j^4 \leq 3.
\end{equation*}
Using the Paley-Zygmund inequality, which asserts that if $Y\ge 0$ is a random variable with $0\neq \mathbb{E}(Y^2)<\infty$, then 
$$\mathbb{P}(Y>\theta \cdot \mathbb{E}(Y))\ge (1-\theta)^2 \cdot \frac{\mathbb{E}(Y)^2}{\mathbb{E}(Y^2)} \quad \mbox{for }0\le \theta\le 1,$$
we obtain
$$\mathbb P\left(f(\epsilon_1,\dots,\epsilon_n) \geq \frac12 \cdot \mathbb E(f(\epsilon_1,\dots,\epsilon_n)) \right) \geq \frac14 \cdot \frac{\mathbb E(f(\epsilon_1,\dots,\epsilon_n))^2}{\mathbb E(f(\epsilon_1,\dots,\epsilon_n)^2)} \geq \frac1{12}.$$ This proves the claim.
\end{proof}
  
\begin{proof}[Proof of Proposition \ref{prop: 3.3 (d)}]
Fix $\xi_1,\dots,\xi_n\in H$ with $\sum_{j=1}^n\|\xi_j\|^2\le 1$. 
By Lemma \ref{lem: boundedness in measure}, there exists a non-increasing function $C' \colon (0,\infty) \to [0,1]$, such that $\lim_{ \lambda \to \infty} C'(\lambda)= 0$ satisfying
$$m\left(\left\{x \in X; |(T\xi)(x)| \geq \lambda \right\}\right) \leq C'(\lambda)$$
for all $\xi \in H$ with $\|\xi\| \leq 1$.

For any constant $\alpha > 0$, we compute 
\begin{eqnarray}
&& \int_X\mathbb{P}\left(\left(\sum_{j=1}^n\epsilon_jT(\xi_j)(x)\right)^2\ge \alpha\right)\,\text{d}m(x) \notag \\
&\ge &\int_{\left\lbrace x \in X\mid \frac12 \cdot\sum_{j=1}^n T(\xi_j)(x)^2\ge \alpha \right\rbrace} \mathbb{P}\left(\left(\sum_{j=1}^n\epsilon_j T(\xi_j)(x)\right)^2\ge \frac12 \cdot\sum_{j=1}^nT(\xi_j)(x)^2\right) \,\text{d}m(x)\notag \\
&\ge&\frac1{12} \cdot m \left( \left\lbrace x \in X; \frac12 \cdot\sum_{j=1}^n (T\xi_j(x))^2\ge \alpha \right\rbrace\right)=:\frac{p(\alpha)}{12}.\label{eq: p(alpha)}
\end{eqnarray}
Here in the last inequality we used Lemma \ref{lem: KhinchinL0}. 
For each $\omega \in \Omega$, define 
$$A(\omega):=\left \{x\in X; \left (\sum_{j=1}^n\epsilon_j(\omega)(T\xi_j)(x)\right )^2\ge \alpha\right \}\subset X,$$ and set $B:=\{\omega\in \Omega; m(A(\omega))\ge \frac{p(\alpha)}{24}\}$. If $\mathbb{P}(B)<\frac{p(\alpha)}{24}$, then by Fubini Theorem,
\eqa{
\int_X \mathbb{P}\left (\left (\sum_{j=1}^n\epsilon_j(T\xi_j)(x)\right )^2\ge \alpha\right )\,\text{d}m(x)&\le \mathbb{P}(B)+\int_X\mathbb{P}\left (\omega\in B^{\text{c}}; x\in A(\omega)\right )\,\text{d}m(x)\\
&< \frac{p(\alpha)}{24}+\int_{B^{\text{c}}}\int_{A(\omega)}\,\text{d}m(x)\,\text{d}\mathbb{P}(\omega)\\
&<\frac{p(\alpha)}{12},
}
which contradicts (\ref{eq: p(alpha)}). Therefore $\mathbb{P}(B)\ge \frac{p(\alpha)}{24}$ holds. 
Since $H$ is a Hilbert space, we have the equality
$$
\mathbb E \left(\left\lVert\sum_{j=1}^n\epsilon_j\xi_j\right\rVert^2 \right)= \sum_{j=1}^n\lVert \xi_j\rVert^2\le1.
$$
Therefore by a similar argument, we can show that the set $C:=\{\omega \in \Omega;\ \|\sum_{j=1}^n\varepsilon_j(\omega)\xi_j\|^2\le \frac{48}{p(\alpha)}\}$ satisfies $\mathbb{P}(C)\ge 1-\frac{p(\alpha)}{48}$. In particular, $\mathbb{P}(B\cap C)>0$ holds (if $p(\alpha)=0$, then $C=\Omega$, in which case the conclusion is still true). 
 Let $\omega \in B\cap C\neq \emptyset$ and set  $\xi = \sum_{j=1}^n \epsilon_j(\omega)\xi_j$. Then 
$\|\xi\|^2\le \frac{48}{p(\alpha)}$, so that  
\begin{eqnarray*}
p(\alpha)=m\left(\left\{x \in X;\, \sum_{j=1}^n (T\xi_j(x))^2\ge 2\alpha \right\} \right)&\le& 24 \cdot m( \{x \in X ; (T\xi(x))^2 \geq \alpha \}) \\
&\leq& 24 \cdot C'\left (\left (\frac{\alpha \cdot p(\alpha)}{48}\right )^{\frac{1}{2}}\right )
\end{eqnarray*}
We apply this now with $\lambda^2=2 \alpha$. 
If $p(\alpha)=p(\frac{\lambda^2}{2})>\lambda^{-1}$, then $\alpha p(\alpha)>\frac{\lambda}{2}$ and since $C'(\cdot)$ is non-increasing, we have 
$$p(\alpha)=m\left (\left \{x\in X; \sum_{j=1}^n(T\xi_j(x))^2\ge \lambda^2\right \}\right )\le 24C'\left (\left (\tfrac{\lambda}{96}\right )^{\frac{1}{2}}\right )$$
We can now set $C(\lambda):= \max \left\{ \lambda^{-1}, 24C'\left (\left (\tfrac{\lambda}{96}\right )^{\frac{1}{2}}\right )\right\}$. This finishes the proof.
\end{proof}
We are now ready to prove the existence of an equivariant Maurey--Nikishin factorization. 
\begin{proof}[Proof of Theorem \ref{thm: Gamma-inv Nikisihin}] 
This part of the proof is exactly the same as Maurey's argument given as in \cite[Theorem 3.3]{GCRF85}. We include the proof for the reader's convenience. Define a subset $A'$ of $A$ (see (\ref{eq: def of A})) by 
$$A':=\left\{ \sum_{j=1}^n |(T\xi_j)|^2; n \in \mathbb N, \xi_1,\dots,\xi_n \in H, \sum_{j=1}^n \|\xi_j\|_{\pi}^2 \leq 1 \right\}\subset L^0(X,m)_+,$$
where $\|\xi\|_{\pi}:=\sup_{s\in \Gamma}\|\pi(s)\xi\|\ (\xi\in H)$. Then since $\|\pi(s)\xi\|_{\pi}=\|\xi\|_{\pi}\ (s\in \Gamma,\,\xi\in H)$, by  (\ref{eq: T(pi(g))}), the set $A'$ is globally $\Gamma$-invariant and convex.  Since $T$ is continuous and using Proposition \ref{prop: 3.3 (d)}, a very similar argument as in the proof of Lemma \ref{lem: boundedness in measure} (ii)$\Rightarrow $(iii) implies the following: for every $t>0$, there exists $K(t)>0$ such that for any finite sequence $\xi_1,\dots,\xi_n\in H$, one has (using $(\psi^2)^{\star}=(\psi^{\star})^2)$)
\begin{equation*}
\left (\sum_{j=1}^n(T\xi_j)^2\right )^{\star}(t)\le K(t)^2\sum_{j=1}^{n}\|\xi_j\|^2.
\end{equation*}By Proposition \ref{prop: GCRF3.1}, for every $\varepsilon>0$, there exists a $\Gamma$-invariant measurable subset $S(\varepsilon)\subset X$ with $m(X\setminus S(\varepsilon))<\varepsilon$ satisfying (\ref{eq: sup_varphi}) with $A$ replaced by $A'$.
In particular, this shows by (\ref{eq: sup_varphi}) that for every $\varphi\in A'$,  $$\int_{S(\varepsilon)}\varphi(x)\,\text{d}m(x)\le 2K\left(\tfrac{\varepsilon}{2}\right)^2$$ holds. Therefore for every $\xi\in H$, $\|\xi\|_{\pi}^2\le \|\pi\|^2\|\xi\|^2$ (where $\|\pi\|:=\sup_{s\in \Gamma}\|\pi(s)\|$), so that we have  
\begin{equation*}
\int_{S(\varepsilon)}[T\xi](x)^2\ \text{d}m(x)\le 2K(\tfrac{\varepsilon}{2})^2\|\pi\|^2\|\xi\|^2.
\end{equation*}
Set $C_\varepsilon:=2K\left(\frac{\varepsilon}{2} \right)^2\|\pi\|^2>0$. 
Then let
\begin{equation*}
\varphi(x):=\sum_{j=1}^{\infty}k_j1_{S(1/j)}(x),\ \ \ \ \ x\in X,
\end{equation*} 
where $k_j>0$ and $\sum_{j=1}^{\infty}k_jC_{1/j}=1$. Since $C_{\varepsilon}$ is non-increasing in $\varepsilon$, $k_j\stackrel{j\to \infty}{\to}0$, which implies that $\varphi\in L^{\infty}(X,m)$. Also, $\varphi$ is $m$-a.e. positive and $\Gamma$-invariant. Moreover we have:
\eqa{
\int_X \varphi (x)[T\xi](x)^2\ \text{d}m(x)&=\sum_{j=1}^{\infty}k_j\int_{S_{1/j}}[T\xi](x)^2\ \text{d}m(x)\\
&\le \sum_{j=1}^{\infty}k_jC_{1/j}\|\xi\|^2=\|\xi\|^2.
} 
This finishes the proof. 
\end{proof}
Recall from the non-commutative integration theory that a self-adjoint operator $A\in \widetilde{M}$ is square-integrable, if 
\[\|A\|_2^2:=\int_{\mathbb{R}}\lambda^2\,\text{d}(\tau(E_A(\lambda))<\infty.\]
Here, $A=\int_{\mathbb{R}}\lambda\,\text{d}E_A(\lambda)$ is the spectral decomposition of $A$. 
Then $T\in \widetilde{M}$ is called square-integrable if $|T|=(T^*T)^{\frac{1}{2}}$ is square-integrable. 
The set of all square-integrable operators can naturally be identified with $L^2(M)$ (see \cite{Nelson74} for details). 
We need the following simple but useful lemma. 
\begin{lemma}\label{lem: L^2homeo}
Let $S\colon H\to \widetilde{M}$ be an $\mathbb{R}$-linear homeomorphism onto its range. If $S(\xi)\in L^2(M)$ for every $\xi\in H$, then regarded as $S\colon H\to L^2(M)$, $S$ is an $\mathbb{R}$-linear homeomorphism onto its range, where $L^2(M)$ is equipped with the $L^2$-topology (instead of the $\tau$-measure topology).   
\end{lemma}
\begin{proof}
We first show that in $L^2(M)$, the $L^2$-topology is stronger than the $\tau$-measure topology. 
Let $A_n\in L^2(M)\ (n\in \mathbb{N})$ be such that $\|A_n\|_2\stackrel{n\to \infty}{\to} 0$. 
Then 
\eqa{
\left \|(A_n-i)^{-1}-(0-i)^{-1}\right \|_2^2&=\int_{\mathbb{R}}\left |\frac{1}{\lambda-i}-\frac{1}{0-i}\right |^2\,\text{d}\tau(E_{A_n}(\lambda))\\
&=\int_{\mathbb{R}}\frac{\lambda^2}{\lambda^2+1}\,\text{d}\tau(E_{A_n}(\lambda))\\
&\le \int_{\mathbb{R}}\lambda^2\,\text{d}\tau(E_{A_n}(\lambda))\\
&=\|A_n\|_2^2\stackrel{n\to \infty}{\to}0.
}
This shows the claim. Note also that this proves that $S^{-1}\colon S(H)\ni S(\xi)\to \xi\in H$ is continuous: 
\eqa{
\|S(\xi_n)\|_2\stackrel{n\to \infty}{\to}0&\Rightarrow S(\xi_n)\stackrel{n\to \infty}{\to} 0\ \text{SRT}\\
&\Leftrightarrow \|\xi_n\|\stackrel{n\to \infty}{\to} 0.
}
Next, we show that $S\colon H\to L^2(M)$ is $L^2$-continuous. 
Since $H, L^2(M)$ are Banach spaces, it suffices to show that $S$ is closed, by the Closed Graph Theorem. 
So assume that $\xi_n\in H$ converges to 0, and $S(\xi_n)$ converges in $L^2$ to some $A\in L^2(M)$. Then $\|\overline{S(\xi_n)-A}\|_2\stackrel{n\to \infty}{\to} 0$, whence by the first part, we also have $\overline{S(\xi_n)-A}\stackrel{n\to \infty}{\to} 0$ in SRT. Since SRT is a vector space topology in $\widetilde{M}$ (because $M$ is finite), we obtain $S(\xi_n)\stackrel{n\to \infty}{\to} A$ in SRT.  
However, $S$ is continuous in SRT, whence $S(\xi_n)\stackrel{n\to \infty}{\to} 0$ in SRT. Since SRT is Hausdorff, $A=0$. Therefore $S$ is closed, whence $L^2$-continuous.  
\end{proof}
\begin{remark} 
In the above lemma, we did not claim that in $L^2(M)$, SRT and the $L^2$-topology are the same. In fact the two topologies are different (but they agree on the image of $S$, i.e., $S$ defines a strong embedding in the sense of \cite[Definition 6.4.4]{KaltonAlbiac06}).  
To see this in the commutative case, let $(\Omega,\mu)$ be a diffuse probability space, and choose a sequence $(B_n)_{n=1}^{\infty}$ of measurable subsets of $\Omega$ such that $0<\mu(B_n)\stackrel{n\to \infty}{\to}0$. 
Define $f_n\in L^2(\Omega,\mu)$ by 
\[f_n(x)=\frac{1}{\mu(B_n)^{\frac{1}{2}}}1_{B_n}(x),\ \ \ \ \ x\in \Omega.\]
Then $\|f_n\|_2^2=1$ for $n\in \mathbb{N}$. 
On the other hand, 
\eqa{
\|(f_n-i)^{-1}-(0-i)^{-1}\|_{L^2(\Omega,\mu)}^2
&=\int_{B_n}\frac{1}{1+\mu(B_n)}\,\text{d}\mu(x)\\
&\le \int_{B_n}1\,\text{d}\mu(x)=\mu(B_n)\stackrel{n\to \infty}{\to}0
}
This shows that $f_n\stackrel{n\to \infty}{\to}0$ (SRT) in $L^0(\Omega,\mu)$, but $\|f_n\|_2=1\not\to 0$. 
\end{remark}
\begin{proof}[Proof of Theorem \ref{thm: keythm}]
By Theorem \ref{thm: Gamma-inv Nikisihin}, there exists an $m$-a.e.\ positive $\varphi\in L^0(X,m)$ such that 
(\ref{eq: Equivariant Nikishin}) holds. We define an $\mathbb{R}$-bilinear form $\nai{\cdot}{\cdot}'_{\pi}\colon H\times H\to \mathbb{R}$ by 
\begin{equation*}
\nai{\xi}{\eta}'_{\pi}:=\int_X \varphi(x)[T\xi](x)[T\eta](x)\,\text{d}m(x),\ \ \ \ \xi,\eta\in H. 
\end{equation*}
If $\nai{\xi}{\xi}_{\pi}'=0$, then $\varphi |T\xi|^2=0$ a.e., and since $\varphi>0$ a.e., $T\xi=0$ a.e. But since $T$ is injective, we get $\xi=0$. Thus this is an inner product. Let $s\in \Gamma$. 
Then for $\xi,\eta\in H$, the $\Gamma$-invariance of $\varphi$ and $m$ show that 
\eqa{
\nai{\pi(s)\xi}{\pi(s)\eta}_{\pi}'&=\int_X \varphi(x)[T\xi](s^{-1}x)[T\eta](s^{-1}x)\ \text{d}m(x)\\
&=\int_X\varphi(sx)[T\xi](x)[T\eta](x)\ \text{d}m(sx)\\
&=\int_X\varphi(x)[T\xi](x)[T\eta](x)\ \text{d}m(x).
}
Therefore $\nai{\cdot}{\cdot}_{\pi}'$ is $\Gamma$-invariant. We show that this inner product generates the same topology on $H$ as the original norm. 
Define $S\colon H\to L^2(X,m)\subset L^2(M)$ by $S(\xi):=\varphi^{1/2}T(\xi)=M_{\varphi^{1/2}}T(\xi)$. Note that the multiplication operator $M_{\varphi^{1/2}}$ is a (possibly unbounded) operator affiliated with $M$. 
Therefore, the function $$M_{\varphi^{1/2}}\colon \widetilde{M}\ni X\mapsto \varphi^{1/2}X\in \widetilde{M}$$ is continuous with continuous inverse $M_{\varphi^{-1/2}}$, so it is a homeomorphism. Therefore we obtain that $S\colon H\to \widetilde{M}$ is also an $\mathbb{R}$-linear homeomorphism onto its range. Moreover, we have the inclusion $S(H)\subset L^2(X,m)\subset L^2(M)$. Thus by Lemma \ref{lem: L^2homeo}, $S\colon H\to L^2(M)$ is a homeomorphism onto its range where $L^2(M)$ is equipped with the $L^2$-topology.  
Let $(\xi_n)_{n=1}^{\infty}$ be a sequence in $H$. 
Then 
\eqa{
\|\xi_n\|\stackrel{n\to \infty}{\to}0&\Leftrightarrow \|S(\xi_n)\|_{L^2(M)} =\|\xi\|_{\pi}'\stackrel{n\to \infty}{\to}0.
}
This shows that $\|\cdot\|_{\pi}'$ generates the same topology as $\|\cdot\|$. 
Finally, we complexify the inner product. Define a new sesqui-linear form $\nai{\cdot}{\cdot}_{\pi}\colon H\times H\to \mathbb{C}$ by
\begin{equation}
\nai{\xi}{\eta}_{\pi}:=\frac{1}{4}\{\nai{\xi}{\eta}_{\pi}'+\nai{i\xi}{i\eta}_{\pi}'+i\nai{\xi}{i\eta}_{\pi}'-i\nai{i\xi}{\eta}_{\pi}'\},\ \ \ \ \xi,\eta\in H.\label{eq: complexify inner product}
\end{equation}
Since $\nai{\cdot}{\cdot}_{\pi}'$ is $\mathbb{R}$-bilinear, so is $\nai{\cdot}{\cdot}_{\pi}$. But if $\xi,\eta\in H$, then 
\eqa{
\nai{i\xi}{\eta}_{\pi}&=\frac{1}{4}\{\nai{i\xi}{\eta}_{\pi}'+\nai{-\xi}{i\eta}_{\pi}'+i\nai{i\xi}{i\eta}_{\pi}'+i\nai{\xi}{\eta}_{\pi}'\}\\
&=\frac{1}{4}\{i\nai{\xi}{\eta}_{\pi}'+i\nai{i\xi}{i\eta}_{\pi}'-\nai{\xi}{i\eta}_{\pi}'+\nai{i\xi}{\eta}_{\pi}'\}\\
&=i\nai{\xi}{\eta}_{\pi},\\
\nai{\xi}{i\eta}'&=\frac{1}{4}\{\nai{\xi}{i\eta}_{\pi}'-\nai{i\xi}{\eta}_{\pi}'-i\nai{\xi}{\eta}_{\pi}'-i\nai{i\xi}{i\eta}_{\pi}'\}\\
&=\frac{1}{4}\{-i\nai{\xi}{\eta}_{\pi}'-i\nai{i\xi}{i\eta}_{\pi}'+\nai{\xi}{i\eta}_{\pi}'-\nai{i\xi}{\eta}_{\pi}'\}\\
&=-i\nai{\xi}{\eta}_{\pi}.
}
Thus $\nai{\cdot}{\cdot}_{\pi}$ is sesqui-linear. 
Note that 
\[\|\xi\|_{\pi}^2=\frac{1}{4}(\|\xi\|_{\pi}')^2+\frac{1}{4}(\|i\xi\|_{\pi}')^2,\]
so that $\nai{\cdot}{\cdot}_{\pi}$ also generates the same Hilbert space topology. 
Since $\pi(s)\ (s\in \Gamma)$ is $\mathbb{C}$-linear and preserves $\nai{\cdot}{\cdot}_{\pi}'$, it is also clear that $\pi(s)$ preserves $\nai{\cdot}{\cdot}_{\pi}$ as well. By Lemma \ref{lem: Dixmier}, this shows that $\pi$ is unitarizable. 
\end{proof}

Note that in our construction $G=H\rtimes_{\pi}\Gamma$, it is essential that $\Gamma$ is discrete. It is therefore of interest to consider the following question: 
\begin{question}
Is there a connected, unitarily representable SIN Polish group which is not of finite type? 
\end{question}
Note also that in the proof of Theorem \ref{thm: keythm}, we have shown the following result, which is of independent interest:
\begin{corollary}\label{cor: Cor to MaureyNikishin} Let $H$ be a separable real Hilbert space, and let $f\in \mathcal{P}(H)$. Then there exist 
\begin{itemize}
\item[{\rm{(i)}}] a compact metrizable space $X$ and a Borel probability measure $m$ on $X$,
\item[{\rm{(ii)}}] a bounded operator $T\colon H\to L^2(X,m;\mathbb{R})$, and 
\item[{\rm{(iii)}}] a positive measurable function $\varphi\in L^{0}(X,m)$, 
\end{itemize}
such that 
\begin{equation*}
f(\xi)=\int_Xe^{i\varphi(x)[T\xi](x)}\ {\rm{d}}m(x),\ \ \  \ \ \ \xi\in H.\end{equation*}
Moreover:
\begin{itemize}
\item[{\rm{(iv)}}] If $f$ generates a neighborhood basis of $0$, then $T$ can be chosen to be a real-linear homeomorphism onto its range. 
\item[{\rm{(v)}}] We may arrange $T$ and $\varphi$ such that if $\mathcal{G}$ is a uniformly bounded subgroup of the group $\{u\in {\rm{GL}}(H); f(u\xi)=f(\xi)\ (\xi\in H)\}$ of invertible operators on $H$ fixing $f$ pointwise, then there exists a continuous homomorphism $\beta$ from $\mathcal{G}$ to the group ${\rm{Aut}}(X,m)$ of all $m$-preserving automorphisms of $X$ with the weak topology, such that 
\begin{equation*}
[T(u\xi)](x)=\beta(u)[T(\xi)](x),\ \ \ \varphi(\beta(u)x)=\varphi(x),\ \ \ \ \ u\in \mathcal{G},\ \xi\in H,\ m\text{-a.e.\ }x\in X.
\end{equation*} 
\end{itemize}
\end{corollary}
One may regard Corollary \ref{cor: Cor to MaureyNikishin} as an equivariant version of the classical result (cf. \cite[Lemma 4.2]{AMM85}) that any $f\in \mathcal{P}(H)$ is of the form 
\begin{equation*}
f(t\xi)=\int_Xe^{it[U\xi](x)}\,\text{d}m(x),\ \ \ \ \xi\in H,\ t\in \mathbb{R}
\end{equation*}
for a probability space $(X,m)$ and a continuous linear map $U\colon H\to L^0(X,m)$. 

\begin{remark}
A typical example of the above corollary is the positive definite function $f(\xi)=e^{-\|\xi\|^2/4}$.
To see this, let us fix an orthogonal basis $\{\xi_i\}$ for $H$.
In this case, we can choose $X=(\dot{\mathbb{R}})^{\mathbb{N}}$ 
the product of an infinite number of copies of the one-point compactification of $\mathbb{R}$,
$m=$ the infinite dimensional Gaussian measure,
$T(\xi_i):X\ni x\mapsto x_i\in\dot{\mathbb{R}}$
and $\varphi=1$.
The triple $(X,m,T)$ is called a Gaussian random process indexed by $H$.
For more detatils, see \cite{simon}.
\end{remark}

\begin{remark}
It is clear from the proof that it applies whenever $H$ is a Banach space of Rademacher type $2$ (see \cite[$\S$6]{KaltonAlbiac06} for definitions). Thus we have shown as a corollary that a Banach space $B$ is isomorphic (as a Banach space) to a Hilbert space if and only if a positive definite function generates the topology of $B$ and is of type $2$. Moreover, the isomorphism can be made equivariant with respect to a group of symmetries that preserve the positive definite function. One may compare this result with Kwapie\'n's Theorem \cite{Kwapien72} that $B$ is isomorphic to a Hilbert space if and only if $B$ is of type 2 and cotype 2. Note that the Kwapie\'n Theorem does not produce an equivariant isomorphism but requires less in the sense that any Banach space isomorphic to a subspace of $L^0(X,m)$ is of cotype 2, but not all cotype 2 Banach spaces can be isomorphic to a subspace of $L^0(X,m)$ (see the remark after \cite[Corollary 8.17]{BenyaminiLindenstrauss00} and \cite{Pisier78}). 

\end{remark}
\section*{Appendix}

The following two results are well-known, but we include the proofs for the reader's convenience. 
\begin{proposition}\label{prop: trace is integration}
Let $A$ be a commutative von Neumann algebra with separable predual, and let $\tau$ be a normal faithful tracial state on $A$. Let $G$ be a Polish group, and let $\beta\colon G\to {\rm{Aut}}(A)$ be a continuous $\tau$-preserving action.  
Then there exists a separable compact metrizable space $X$, a Borel probability measure $m$ on $X$, an $m$-preserving continuous action $G\times X\ni (g,x)\mapsto gx\in X$ and a $*$-isomorphism $\Phi\colon A\to L^{\infty}(X,m)$ such that 
\begin{align}
\tau(a)&=\int_X\Phi(a)(x)\,{\rm{d}}m(x),\ \ \ \ a\in A,\label{eq: trace = integration}\\
\Phi(\beta_g(a))(x)&=\Phi(a)(g^{-1}x),\ \ \ \ g\in G,\,\mu{\rm{-a.e.}}\, a\in A.\label{eq: covariance of action}
\end{align}
Moreover, the $G$-action is continuous as a map $G\to {\rm{Aut}}(X,m)$, where the group ${\rm{Aut}}(X,m)$ of all $m$-preserving automorphism of $X$ is equipped with the weak topology. 
\end{proposition}
\begin{lemma}\label{lem: extension}
Let $M,N$ be von Neumann algebras, and let $\psi,\varphi$ be normal faithful states on $M,N$, respectively. Let $M_0$ (resp. $N_0$) be a $*$-strongly dense subalgebra of $M$ (resp. $N$). 
 If $\alpha\colon M_0\to N_0$ is a $*$-isomorphism such that $\varphi(\alpha(x))=\psi(x)$ for every $x\in M_0$, then $\alpha$ can be extended to a $*$-isomorphism from $M$ onto $N$ such that $\varphi (\alpha(x))=\psi(x)\ (x\in M)$. 
\end{lemma}
\begin{proof}
By the GNS construction, we may assume that $M$ acts on $L^2(M,\psi)$ and $N$ acts on $L^2(N,\varphi)$ so that $\psi=\nai{\ \cdot\ \xi_{\psi}}{\xi_{\psi}}$, $\varphi=\nai{\ \cdot\ \xi_{\varphi}}{\xi_{\varphi}}$ for unit vectors $\xi_{\psi}\in L^2(M,\psi)$ and $\xi_{\varphi}\in L^2(N,\varphi)$.  Define $U_0\colon M_0\xi_{\psi}\to N_0\xi_{\varphi}$ by 
\[U_0x\xi_{\psi}:=\alpha(x)\xi_{\varphi},\ \ \ \ \ x\in M_0.\]
Since $\psi$ is faithful, $U_0$ is well-defined. Also, since $M_0$ is $*$-strongly dense in $M$ and $\xi_{\psi}$ is cyclic for $M$, $M_0\xi_{\psi}$ is dense in $L^2(M,\psi)$. Similarly, $N_0\xi_{\varphi}$ is dense in $L^2(N,\varphi)$. By $\varphi\circ \alpha=\psi$ on $M_0$, we have 
\[\|U_0x\xi_{\psi}\|^2=\varphi(\alpha(x^*x))=\psi(x^*x)=\|x\xi_{\psi}\|^2,\ \ \ \ x\in M_0.\]
Therefore, $U_0$ extends to an isometry $U$ from $L^2(M,\psi)$ into $L^2(N,\varphi)$ and the range of $U$ contains a dense subspace $N_0\xi_{\varphi}$. Therefore $U$ is onto. If $x\in M_0$, then for each $y\in N_0$, 
\[UxU^*y\xi_{\varphi}=Ux\alpha^{-1}(y)\xi_{\psi}=\alpha(x)y\xi_{\varphi},\]
so that by the density of $N_0\xi_{\varphi}$, we have 
\[UxU^*=\alpha(x),\ \ \ \ x\in M_0.\]
Let now $x\in M$. Then there exists a net $(x_i)_{i\in I}$ in $M_0$ converging $*$-strongly to $x$. Then $Ux_iU^*=\alpha(x_i)\ (\in N)$ converges $*$-strongly to $UxU^*$. 
Therefore $UxU^*\in N$. Then it is clear that $M\ni x\mapsto UxU^*\in N$ defines a $*$-isomorphism with inverse $N\ni y\mapsto U^*yU\in M$. The uniqueness of the extension is clear by the $*$-strong density of $M_0$ and the fact that isomorphisms between von Neumann algebras are normal. It is now clear that $\varphi\circ \alpha=\psi$ on $M$. This finishes the proof. 
\end{proof}
\begin{proof}[Proof of Proposition \ref{prop: trace is integration}]
Let $A_0:=\{x\in A; \lim_{g\to 1}\|\alpha_g(x)-x\|=0\}$, which is a norm-dense C$^*$-subalgebra of $A$. 
Let $\{a_n\}_{n=1}^{\infty}$ be a $*$-strongly dense subset of $A_0$, and let $C=C^*(\{\beta_g(a_n);g\in G,\ n\in \mathbb{N}\}\cup \{1\})$. Then since $G$ is separable and $g\mapsto \beta_g(a_n)$ is norm-continuous for all $n\in \mathbb{N}$, $C$ is a norm-separable $*$-strongly dense C$^*$-subalgebra of $A$, and let $X=\text{Spec}(C)$ (the Gelfand spectrum), which is a separable and compact metrizable space, and let $\Phi\colon C\to C(X)$ be the Gelfand transform. 
Since $\tau\circ \Phi^{-1}\in C(X)^*_+$, there exists a Borel probability measure $m$ on $X$ such that 
\[\tau(c)=\int_X\Phi(c)(x)\,\text{d}m(x),\ \ \ \ c\in C.\]
We let $C(X)$ act on $L^2(X,m)$ by multiplication, so $C(X)\subset L^{\infty}(X,m)$ is a $*$-strongly dense $*$-subalgebra of $L^{\infty}(X,m)$. The map $\tau_{m}\colon L^{\infty}(X,m)\ni f\mapsto \int_Xf\,\text{d}m\in \mathbb{C}$ is a normal faithful tracial state on $L^{\infty}(X,m)$. Moreover, $\tau_{m}(\Phi(c))=\tau(c)$ for every $c\in  C$. Therefore by Lemma \ref{lem: extension}, $\Phi$ can be extended to a $*$-isomorphism $\Phi\colon C''=A\to C(X)''=L^{\infty}(X,m)$ such that $\tau_{m}(\Phi(a))=\tau(a)$ for every $a\in A$. This shows (\ref{eq: trace = integration}). Now let $g\in G$ and $x\in X$. Let $m_x\in C(X)^*$ be the delta probability measure at $x\in X$. 
Let $\beta^*$ be the dual action of $G$ on $C(X)^*$ given by 
\begin{equation}
\nai{\beta_g^*(\nu)}{f}:=\nai{\nu}{\Phi\circ \beta_{g^{-1}}\circ \Phi^{-1}(f)},\ \ \ \ \ \ \nu\in C(X)^*,\ f\in C(X),\ g\in G.
\end{equation}
Since $\beta$ is an action, $\beta_g^*$ maps pure states (on $C\cong C(X)$) to pure states, and pure states on $C(X)$ are precisely the delta probability measures, there exists a unique point $g\cdot x\in X$ such that 
\[\beta_g^*(m_x)=m_{g\cdot x}.\]
We show that $G\times X\ni (g,x)\mapsto g\cdot x\in X$ is a continuous $m$-preserving action. 
It is routine to check that this is indeed an action of $G$ on $X$. 
%(e.g., $\beta_{g_1g_2}^*(\mu_x)=\mu_{(g_1g_2)x}, \beta_{g_1}^*\circ \beta_{g_2}^*(\mu_x)=\beta_{g_1}^*(\mu_{g_2\cdot x})=\mu_{g_1(g_2x)}$, etc). 
We show that the action is continuous. Since $G$ is a Polish group and $X$ is a Polish space, it suffices to show that the action is separately continuous (see e.g., \cite[Theorem 9.14]{Kechris96}). Suppose that $(x_n)_{n=1}^{\infty}\subset X$ converges to $x\in X$ and let $g\in G$. Then by the definition of the Gelfand spectrum, $m_{x_n}\stackrel{n\to \infty}{\to}m_x$ (weak*). Since $\beta_g$ induces an action on $C$ ($C$ is $\beta$-invariant by construction), we have 
$$\beta_g^*(m_{x_n})=m_{g\cdot x_n}\stackrel{n\to \infty}{\to} \beta_g^*(m_x)=m_{g\cdot x},$$
which implies $g\cdot x_n\stackrel{n\to \infty}{\to} g\cdot x$ (again by the definition of the topology on the Gelfand spectrum). Let now $x\in X$ and $(g_n)_{n=1}^{\infty}\subset G$ be a sequence converging to $g\in G$. 
Then for each $f\in C(X)$, 
\eqa{
\nai{\beta_{g_n}^*(m_x)}{f}&=\nai{m_x}{\Phi\circ \beta_{g_n^{-1}}\circ \Phi^{-1}(f)}\stackrel{n\to \infty}{\to}\nai{m_x}{\Phi\circ \beta_{g^{-1}}\circ \Phi^{-1}(f)}\\
&=\nai{\beta_g^*(m_x)}{f},
} 
whence $g_n\cdot x\stackrel{n\to \infty}{\to}g\cdot x$. Therefore the $G$-action on $X$ is continuous. 
Next, we show that for each $g\in G$, the map $X\ni x\mapsto g\cdot x\in X$ preserves $m$. It suffices to show that 
\begin{equation*}
\int_Xf(g^{-1}x)\,\text{d}m(x)=\int_Xf(x)\,\text{d}m(x),\ \ \ \ f\in C(X).
\end{equation*}
Let $\nu=\sum_{i=1}^n\lambda_im_{x_i}\in \text{conv}\{m_x;x\in X\}=:\mathcal{D}$. 
Then 
\eqa{
\beta_g^*(\nu)=\sum_{i=1}^{n}\lambda_i\delta_{g\cdot x_i}=\nu(g^{-1}\ \cdot\ ).
}
By the Krein-Milman Theorem, $\mathcal{D}$ is weak*-dense in the state space of $C(X)$. This shows that 
\begin{equation}
\beta_g^*(\nu)=\nu(g^{-1}\ \cdot\ ),\ \ \ \ \ \ \nu\in C(X)^*_+.\label{eq: cov for measure}
\end{equation}
In particular, by $\tau\circ \beta_g=\tau$, 
\eqa{
\int_Xf(g^{-1}x)\,\text{d}m(x)&=\int_Xf(x)\,\text{d}m(g(x))\stackrel{(\ref{eq: cov for measure})}{=}\nai{\beta_{g^{-1}}^*(m)}{f}\\
&=\nai{m}{\Phi\circ \beta_g\circ \Phi^{-1}(f)}\\
&=\tau_{m}(\Phi\circ \beta_g\circ \Phi^{-1}(f))=\tau(\beta_g\circ \Phi^{-1}(f))\\
&=\tau(\Phi^{-1}(f))=\tau_{m}(f)\\
&=\int_Xf(x)\,\text{d}m(x).
}
Therefore the $G$-action on $X$ is $m$-preserving. Finally, we show (\ref{eq: covariance of action}). 
Since the map $$\beta'_g\colon L^{\infty}(X,m)\ni f\mapsto f(g^{-1}\cdot\ )\in L^{\infty}(X,m)$$ defines a $\tau_{m}$-preserving action and since $C(X)$ is $*$-strongly dense in $L^{\infty}(X,m)$, it suffices to show that
(\ref{eq: covariance of action}) holds for every $f=\Phi(a)\in C(X)\ (a\in C),\ x\in X$, and $g\in G$. 
But 
\eqa{
f(g^{-1}x)&=\nai{\mu_{g^{-1}\cdot x}}{f}=\nai{\beta_{g^{-1}}^*(\mu_x)}{f}\\
&=\nai{\mu_x}{\Phi\circ \beta_g\circ \Phi^{-1}(f)}\\
&=\Phi\circ \beta_g\circ \Phi^{-1}(f)(x).
}
Finally, assume that $g_n\stackrel{n\to \infty}{\to}g$ in $G$. Then for every measurable set $A\subset X$, we have 
\eqa{
m(g_nA\bigtriangleup gA)&=\int_X|\Phi\circ \beta_{g_n^{-1}}\circ \Phi^{-1}(1_A)(x)-\Phi\circ \beta_{g^{-1}}\circ \Phi^{-1}(1_A)(x)|\,\text{d}m(x)\\
&\le \|\beta_{g_n^{-1}}\circ \Phi^{-1}(1_A)-\beta_{g^{-1}}\circ \Phi^{-1}(1_A)\|_2\stackrel{n\to \infty}{\to}0,
}
because the $u$-topology on ${\rm{Aut}}(A)$ is the same as the topology of pointwise $\|\cdot \|_2$-convergence. Thus the $G$-action defines a continuous homomorphism $G\to {\rm{Aut}}(X,m)$. 
This finishes the proof. 
\end{proof}

\section*{Acknowledgments}
We thank Professors \L ukasz Grabowski, Magdalena Musat, Yuri Neretin, Izumi Ojima and Mikael R\o rdam for helpful discussions on the $\mathscr{U}_{\rm{fin}}$-problem at early stages of the project. The final part of the manuscript was completed while HA was visiting the Mittag-Leffler Institute for the program ``Classification of operator algebras: complexity, rigidity, and dynamics". He thanks the institute and the workshop organizers for the invitation and for their hospitality. 
This research was supported by JSPS KAKENHI 16K17608 (HA), 6800055 and 26350231 (YM), ERC Starting Grant No.\ 277728 (AT) and the Danish Council for Independent Research through grant no. 10-082689/FNU. (AT), from which HA was also supported while he was working at University of Copenhagen. Last but not least, we would like to thank the anonymous referee for suggestions which improved the presentation of the paper.

\end{document}